\documentclass[final]{siamltex}
\usepackage{amsfonts, amsmath, epsfig}
\usepackage{fancyhdr}
\normalbaselineskip

\newcommand{\bq}{\mathbf q }

\newcommand{\but}{\mathbf{\tilde u}}

\newcommand{\ff}{\mathbf f}

\newcommand{\bD}{\mathbf D}

\newcommand{\bU}{\mathbf U}

\newcommand{\mcT}{\mathcal T}

\newcommand{\phI}{\phi_{\textrm {\begin{tiny}I\end{tiny}}}}

\newcommand{\vv}{{\mathbf V}}

\newcommand{\bV}{\mathbf V}
\newcommand{\Om}{{\Omega}}
\newcommand{\pOm}{{\partial \Omega}}

\newcommand{\om}{{\boldsymbol\omega}}

\newcommand{\tr }{{\textrm {tr\,}}}

\newcommand{\etaperp}{\eta_{\perp}}
\newcommand{\etapar}{\eta_{\|}}
\newcommand{\nablax}{{\nabla_{\tiny{\bx}}}}

\newcommand{\by}{{\mathbf y}}

\newcommand{\bg}{{\mathbf g}}

\newcommand{\uu}{{\mathbf U}}

\newcommand{\mue}{{\mu_{\textrm{\begin{tiny}E\end{tiny}}}}}

\newcommand{\norm}{{\big|\!\big|}}
\newcommand{\bn}{{\mathbf n}}

\newcommand{\Cauchy}{\mathcal T}
\newcommand{\calT}{\mathcal T}
\newcommand{\bcalw}{\mathbf{\mathcal W}}
 \newtheorem{example}{Example}[subsection]
\newcommand{\bu}{{\mathbf u}}

\newcommand{\bv}{\mathbf v}
\newcommand{\bx}{{\mathbf x}}

\newcommand{\bw}{{\mathbf w}}

\newcommand{\op}{{\overline{p}}}
\newcommand{\bW}{{\mathbf W}}

\newcommand{\calP}{\mathcal P}

\newcommand{\dive}{\nabla \cdot}

\begin{document}

\title{Effects of permeability and viscosity in linear polymeric gels}
\author{B. Chabaud\thanks{Los Alamos National Laboratory, Computational Physics Division,
Los Alamos, NM 87545 USA} \and M. C. Calderer\thanks{School of Mathematics, University of Minnesota, 
206 Church Street S.E., Minneapolis, MN 55455, USA.}}

\maketitle
\pagenumbering{arabic}

\begin{abstract}
We propose and analyze a mathematical model of the mechanics of gels, consisting of the laws of balance of mass and linear momentum.  We consider a gel to be an immiscible and incompressible mixture 
of a nonlinearly elastic polymer and a fluid. The problems that we study are motivated by  predictions of the life cycle of body-implantable medical devices.  
Scaling arguments suggest neglecting inertia terms, and therefore, we consider the quasi-static approximation to the dynamics.   We  focus on the linearized system about relevant equilibrium solutions, 
and derive sufficient conditions for the solvability of the time dependent problems. These turn out to be  conditions that guarantee local stability of the equilibrium solutions. 
The fact that some equilibrium solutions of interest are not stress free brings additional challenges to the analysis, and, in particular, to the derivation of the energy law of the systems. It also singles
  out the special role of the rotations in the analysis.  From the point of view of applications, we point out that  the conditions that guarantee stability of solutions also provide criteria to select  material parameters for devices. 
   The boundary conditions that we consider are of two types, first displacement-traction conditions for the governing equation of the polymer component,
 and secondly permeability conditions for the fluid equation.   We present a rigorous study of these conditions in terms of balance laws of the fluid across the interface between the gel and
 its environment \cite{chen-mori-micek-calderer2012},   and use it to justify  heuristic permeability formulations found in the literature \cite{DOI02}, \cite{DOI04}. 
  We also consider the cases of  viscous and inviscid solvent, assume Newtonian dissipation for the   polymer component.  We establish existence of weak solutions for the different boundary permeability conditions 
and viscosity assumptions. 
We present two-dimensional, finite element numerical simulations  to study pressure concentration on edges, in connection with the {\it debonding } phenomenon between the gel and the boundary substrate upon 
reaching a critical pressure. 
\end{abstract}
\section{Introduction}
This article addresses mechanical modeling, stability of equilibrium states  and  analysis of boundary value problems of quasi-static  gel dynamics. 
 We assume that a gel is an incompressible and  immiscible mixture of polymer and solvent, and study the  coupled system of equations of balance of mass and linear momentum of the components.
 Boundary conditions are  of traction-displacement type   together with statements of the permeability of the gel boundary to the environmental fluid.   
 This work is motivated by problems arising in the prediction of the  life cycle  of body-implantable medical devices. 
 
 Gels consist of crosslinked or entangled polymeric networks holding fluid. In its swollen state, the polymer confines the solvent and, in turn, the solvent prevents the gel from collapsing into  dry polymer.
 Gels are abundantly  present in nature and occur when materials are placed in a fluid environment.  Devices such as pacemakers, bone replacement units and artificial skin turn into gel when 
 implanted in the body.
 The materials that constitute a device differ in swelling ratio, this causing a build up of  stress at the interfaces, as well as at the contact between the device and its  boundary support. 
High stresses, above the manufacturer's guaranteed   threshold, may cause debonding instability leading to device failure.

 The equations that we analyze encode relevant properties of gel  behavior  such as solvent diffusion, transport of polymer and solvent, friction and viscosity, elasticity, and time relaxation.
 They consist of equations of balance of mass and linear momentum for  the polymer and fluid components, together with the saturation ({\it incompressibility}) constraint. 
  Assuming that the polymer is an isotropic, elastic solid, 
  the total  energy of the  gel  is the sum of the  elastic stored energy function of the polymer and the Flory-Huggins energy of mixing.  The variable fields of the equilibrium  
problem consist of the  volume fractions of the gel components and  the deformation gradient tensor of the polymer. These fields are not all independent, but they are  related by the equation of 
balance of mass of the polymer and the saturation condition on the gel. The problem of constrained energy minimization  studied by  Micek, Rognes and Calderer in \cite{rognes-calderer-micek09} 
established sufficient conditions on the energy and on  the imposed boundary conditions that guarantee existence of a global energy minimizer. 
The authors also developed, analyzed and numerically implemented a mixed finite element discretization of the equilibrium linear operator.  A challenge of the analysis is the presence of residual stress in
 the reference configuration of the polymer. The work yielded numerical evaluations of shear stress at the interface between two gels representing  bone tissue and the artificial implant. 
 However, gel behavior is inherently a time evolution problem due to the combined effects of transport, diffusion and dissipation. This serves as a motivation to  the work presented in this article. 
 
 Scaling arguments for gels consisting of polymer melts  justify neglecting the inertia terms  and analyzing the quasi-static system. 
 Specifically, the size of polymer dissipation effects with respect to inertia results in the latter being dominant at time scales on the order of $10^{-7}$ seconds \cite{CaZha07}.  
Of course, such time scales are negligible for biomedical devices with typical  life-cycle of 20 years.  From a different perspective,  the effect of the inertia terms was analyzed by  
 Zhang and Calderer \cite{CaZha07}. They studied the free boundary problem of gel swelling, in one space dimension, neglecting the Newtonian dissipation of the both gel components. 
  In such a case, the governing equation   turns out to be a   frictional, weakly dissipative,   hyperbolic partial differential equation \cite{Dafermos95}, \cite{Da03}.
 The scaling that justifies retaining the inertia effects    is consistent with the dynamics of  polysaccharide gel  networks found in many living systems, such as in gliding myxobacteria \cite{KAISER03},
 \cite{RD81}.

 The model proposed here shares analogies with deformable porous media flow models but has the additional feature of accounting for interaction between fluid and polymer through the Flory-Huggins energy. 
From a different perspective, especially challenging multi-component mixtures are used in geology and  in  oil and natural gas recovery models  \cite{cushman96a}, \cite{cushman96b}, \cite{cushman00}. In these models, a relevant role is played by the Terzaghi's stress, which is the pressure exerted by the fluid in the pores  against the
stress applied to the rock. Its analog in the case of the gel is the pressure in the solvent accounted by the Flory-Huggins energy. The current analysis is not immediately extendable to the triphasic models when one of the components is compressible.

  The governing system  that we study is motivated by the stress-diffusion coupling model developed by Doi and Yamaue  \cite{DOI02}, \cite{DOI04}, \cite{DOI05}, \cite{DOI06}, 
and also by Hong et al. \cite{suo2}. Feng and He \cite{fenghe} studied purely inviscid gels with impermeable boundary as  governed by the stress-diffusion coupling model. 
 Our model treats the polymer component of the gel as a  nonlinear elastic solid, and includes, both, fluid and polymer dissipation, diffusion, the Flory-Huggins interaction, and accounts for 
different permeability properties of the interface between the gel and the surrounding fluid as well as traction-displacement conditions imposed on the polymer boundary. To our knowledge, all these combined effects have  not been accounted for in the previous works. We  develop an existence theory for the governing system of  partial differential equations, 
linearized about an equilibrium state.
Some of the tools presented  in \cite{fenghe} have been applied to our analysis, in the case that the solvent is inviscid. One novelty of this work is the derivation 
of permeability conditions from  balance balance laws at the interface between the gel and its surrounding fluid, following a model developed for the treatment of
 polyelectrolyte gels \cite{chen-mori-micek-calderer2012}.

 We assume that the polymer component of the gel is an isotropic elastic material with  Newtonian dissipation.  The equations possess bulk, stress free,  equilibrium solutions that are pure expansion or 
compression.  We point out that, as for isotropic elasticity, if the energy is convex with respect to the deformation gradient, these are the only stress free critical points.  However, if the energy  is nonconvex, the system may also admit  non-spherical  equilibrium deformations. These states may be consistent with experimentally observed  
pattern structures  in polyelectrolyte gels \cite{gel-bamboo-pattern03},  \cite{sultan-baodoui-gelbuckling08} and \cite{Onuki89}.
Equilibrium gel states with nonzero stress are also relevant in many applications. In particular, reference configurations with residual stress may be counted as a special case of the former. 
 Our analysis, addresses the two types of linearization of the governing system, first, about stress free spherical deformations, and secondly about 
equilibrium solutions that satisfy mixed traction-displacement boundary conditions.    
Within this perspective, we may consider the process of device implantation as subjecting an originally stress free body  to displacement  initial and boundary conditions, as well as to the environmental
 stress of the surrounding tissue, and the permeability effects of such a contact. The device  
will no longer be at equilibrium under the newly imposed initial and boundary conditions,  and a dynamical process will begin at implantation. The second type of linearization is relevant to  the iterative process of solution of a nonlinear problem.

 We derive restrictions on the constitutive equations that ensure the coercivity of the static operators. These are also known as the Coleman and Noll conditions and guarantee the classical stability requirement that the stress work be non-negative in every strain (\cite{TruesdellNoll2010}, sections 52 and 83). 
 Moreover, we find that the procedure of deriving the energy relation brings out the special role of the rotations in the case that residual stresses are present.  

We assume that the reference configuration of the gel is that of the polymer network previous to the gel formation. Accordingly, the boundary of the current domain is that of the polymer,
 and it evolves with its velocity.
The Eulerian formulation of the governing system of the gel  and the natural Lagrangian setting of solid elasticity of the polymer present challenges to the analyses. These manifest themselves 
in the derivation of the energy relations satisfied by weak solutions of the governing systems. We consider the cases of  impermeable and fully permeable boundary between the gel and  the environmental fluid. 
 The case of a semipermeable gel boundary follows from the former, with some elementary modifications, and so, we omit its presentation. Finally, we point out that the conditions for local equilibria employed 
in the solvability of the time dependent problems are also sufficient to guarantee regularity of the weak solutions. 
 
In related work, we developed and analyzed a numerical method  based on finite elements to simulate solutions of the models presented here,  in two dimensional domains, in the case that the fluid is inviscid. We consider a gel sample in the unit square, subject to zero displacement in two opposite edges and to a fixed pressure of $10^4Pa$ in the other two. We calculate the stress components under the following criteria:
the relative scaling of the Flory-Huggins energy with respect to the elastic one, the degree of stiffness, expansion and compressibility of the polymer, and the type of boundary permeability. 
 The elasticity modulus of the polymer is set at $1$GPa, consistent with values encountered in device materials.  We find that pressure concentrates on the edges where the the displacement is held to zero, its values increasing with decreased compressibility, and large stiffness. Permeability also promotes stress concentration, but with the interior stress being lower than that in the impermeable case.  If the pressure at an edge overcomes the {\it debonding } threshold, then it would detach from its support. The experimental literature reports on values of the debonding pressure for different materials ranging from 0.5 to 10 times the elastic modulus $\mu_E$, when the value of the latter is of the order of $10^7$Pa \cite{shull-creton2004}.

The paper is organized as follows.  In section 2, we present the balance laws of the gel,  the constitutive equations  and discuss the equilibrium states.  
    Section 3  is devoted to the linearization of the governing equations, formulation of boundary conditions, and the study of the local stability of the equilibrium solutions. 
In section 4, we study existence of weak solutions in the case that the fluid is inviscid, and section 5 is devoted to the case of viscous solvent. In both sections, the linearization is carried out
 about uniform dilations or compressions. 
 In section 6,  we derive the energy law in the case that the governing equations are linearized about an arbitrary  equilibrium solution.  This is the main ingredient in extending the analysis of
 sections 4 and 5 to the more general case, and for which we omit the details. The numerical simulations are presented in section 7. 
  Finally, in section 8, we draw some conclusions. 
 This work is  based on the Ph.D dissertation by Brandon Chabaud \cite{chabaud09}.

\section{Modeling of Gel Mechanics}
We assume that a gel is a saturated, incompressible and immiscible mixture of  elastic solid and  fluid.
 In the reference configuration, the polymer occupies a domain $\Omega\subset {\mathbf R}^3$. 
 The solid undergoes a deformation according to the one-to-one,  differentiable  map
\begin{equation}
\by=\by(\bx, t), \quad \textrm {such that\,\,} \det(\nablax\by)>0,\,\, \bx\in\Omega.
\end{equation}
 We let  $\Omega_t=\by(\Omega, t)$ denote  the domain occupied by the gel at time $t\geq 0$, and denote $F=\nablax\by$.  We label  the polymer and fluid components with indices {\it 1} and {\it 2}, respectively.
A point $\by\in\Omega_t$ is occupied by, both, solid and fluid at volume fractions
 $\phi_1=\phi_1(\by,t)$ and $\phi_2=\phi_2(\by,t)$, respectively.  We let $\bv_1=\bv_1(\by,t)$ and $\bv_2=\bv_2(\by, t)$
 denote the corresponding velocity fields.

An {immiscible} mixture is such that  the  constitutive equations  depend explicitly on the volume fractions $\phi_i, i=1,2$.
  We let $\rho_i$  denote the {mass} density of the $i$th component (per unit volume of gel). 
It is  related to the {intrinsic} density, $\gamma_i$, by the equation $\rho_i=\gamma_i \phi_i$, $i=1,2$. Moreover $\gamma_i=\textrm {constant}, i=1,2$  define
 an incompressible mixture.
Throughout this section, and  unless otherwise specified, the $\nabla$ notation refers to derivative with respect to the Eulerian space variable $\by$.
 The assumption of saturation of the mixture, that is, that no species other than polymer and fluid are present,
is expressed by the equation
 \begin{equation}\label{sat}
 \phi_1+\phi_2 = 1.
\end{equation}
(In some terminologies, this  condition is also known  as incompressibility). 
 The governing equations in the Eulerian form  consist of the balance of mass and linear momentum of each component  as well
 as the chain-rule relating the time derivative of the gradient of deformation with the velocity gradient:	
		\begin{eqnarray}
		\frac{\partial \phi_i}{\partial t} +\nabla\cdot(\phi_i\bv_i) &&= 0, \label{mass}\\		
\gamma_i\phi_i\left(\frac{\partial \bv_i}{\partial t} + (\bv_i\cdot\
\nabla)\bv_i\right) &&= \nabla\cdot\calT_i-\phi_i\nabla p- \eta(\bv_i-\bv_j), \,\, 1\leq i\neq j\leq 2  \label{lin-momentum}\\ 
			\frac{\partial F}{\partial t} + (\bv_1\cdot\nabla)F &&=(\nabla\bv_1)F,\label{chainrule}
		\end{eqnarray}
$\eta>0$ constant, $\by\in\Omega_t$, $i=1,2$.  Here  $\Cauchy_i$ is the Cauchy stress tensor of the $i$th component. 
The Lagrangian form of the equation of balance of mass of the polymer component  is 
	\begin{equation}\label{mass-lagrangian}
		\phi_1(\by(\bx,t), t)\,\det F(\bx,t) = \phi_I(\bx),\,\,\,\bx\in\Omega,	\end{equation}
where $0<\phI<1$  represents the prescribed, differentiable,  reference volume fraction.  	Adding up both equations in (\ref{mass}), and  taking (\ref{sat}) into account gives 
	\begin{equation}
	\nabla\cdot(\phi_1\bv_1+ \phi_2\bv_2)=0. \label{div}
	\end{equation}
	The governing system consists of equations   (\ref{sat})-(\ref{chainrule})  supplemented with constitutive equations for $\calT_i$ together with
  prescribed boundary and initial conditions. 
 Rather than analyzing this system directly, we will instead consider the set of equations (\ref{sat}),  (\ref{lin-momentum}), (\ref{chainrule}),
 (\ref{mass-lagrangian}) and (\ref{div}).  
 
\subsection{Gel environment} Upon body implantation, the device becomes immersed in tissue 
 occupying the domain
 $\mathcal B_t$, with $\Omega_t\subset\mathcal B_t$. We let  $\calT_b$ denote the stress in the surrounding fluid.  In \cite{chen-mori-micek-calderer2012}, we prescribe governing equations for the fluid in $\mathcal B_t$ as well as balance laws at the interface $\partial\Omega_t\cap\partial\mathcal B_t$.   In this article, we adopt the simplified assumption that the outside fluid exerts a prescribed pressure $\calT_b=-P_0 I$, and  do not  postulate balance laws  in 
 $\mathcal B_t$.   
Letting $ \bv_b $ denote  the velocity field of the outside fluid,
we assume that the following relations hold on $\partial\Omega_t$ \cite{chen-mori-micek-calderer2012}:
\begin{eqnarray}
 \phi_2(\bv_2-\bv_1)\cdot\bn= && (\bv_b-\bv_1)\cdot\bn:=w, \label{balance-mass-fluid-interface}\\ 
(\bv_b-\bv_1)_{\|}=&& (\bv_2-\bv_1)_{\|}:=\mathbf q, \label{tangent-velocity}\\
(\calT_1+\calT_2)\bn+[p]\bn=&&-\gamma_2w^2(1-\frac{1}{\phi_2})\bn. \label{balance-linmomentum-fluid-interface}
\end{eqnarray}
Equation (\ref{balance-mass-fluid-interface}) is the statement of balance of fluid mass across $\partial\Omega_t$ and (\ref{balance-linmomentum-fluid-interface}) states the balance  of linear
momentum of the fluid across the interface. Here $[p]:= P_0-p$, where $p$ is the gel pressure at the interface limit. It is easy to check that the right hand side of equation (\ref{balance-linmomentum-fluid-interface})
is the change in linear momentum density of fluid crossing a unit area of  the interface, and the left hand side represents the total force per unit area acting on the fluid. 

Also, following \cite{chen-mori-micek-calderer2012}, we assume that the interface has an intrinsic viscosity, with coefficients $\etaperp>0$ and $\etapar>0$, 
respectively, affecting the fluid crossing it in the normal 
direction, or moving tangentially to it. Specifically, we assume that 
                   \begin{eqnarray}
\Pi_\perp:=&&\frac{1}{2}\gamma_2\big((\frac{w}{\phi_2})^2-w^2)-[p]- \bn\cdot(\frac{\calT_2}{\phi_2})\bn=\eta_{\perp} w,\label{piperp} \\
\Pi_\|:=&&(\calT_1\bn)_{\|}=\eta_{\|}(\bv_2-\bv_1)_{\|}. \label{pipar}
\end{eqnarray}
\subsection{Boundary conditions} We assume that the gel is surrounded by its onw fluid. The boundary conditions at the interface $\partial\Omega_t$ consist of the 
set of equations (\ref{balance-mass-fluid-interface})-(\ref{pipar}). In the case that the mass inertia of the fluid is neglected, they yield 
 two types of boundary conditions, traction on the gel and equations expressing the degree of permeability of the gel boundary to its surrounding fluid.
First of all, from (\ref{balance-mass-fluid-interface}) and (\ref{tangent-velocity}), we obtain the velocity $\bv_b$ of the fluid outside the gel but near the boundary:
\begin{equation}
 \bv_b\cdot\bn=(\phi_1\bv_1+\phi_2\bv_2)\cdot\bn,\quad 
 \bv_b\cdot\bq=\bv_2\cdot\bq,\quad \textrm{on}\,\partial\Omega_t,
\end{equation}
for any $\bq\neq 0$ such that  $\bq\cdot\bn=0$. Equation (\ref{balance-linmomentum-fluid-interface}), yields balance of force at the gel-fluid interface, 
\begin{equation}
-P_0\bn=(\calT_1+\calT_2-pI)\bn. \label{interface-traction}
 \end{equation}
Equations (\ref{piperp})-(\ref{pipar}) are statements of semipermeability of the gel interface. Neglecting inertia, and using equation 
(\ref{balance-mass-fluid-interface}),  the first one becomes
\begin{equation}
 -[p]- \bn\cdot(\frac{\calT_2}{\phi_2})\bn=\eta_{\perp} \phi_2(\bv_2-\bv_1)\cdot\bn.\label{semi1}
\end{equation}
Now, substituting equation (\ref{interface-traction}) into (\ref{pipar}), the latter becomes
\begin{equation}
 -\calT_2\bn|_{\|}=\eta_{\|}(\bv_2-\bv_1)_{\|}. \label{semi2}
\end{equation}
So, the boundary conditions on $\partial\Omega_t$ consist of equations (\ref{interface-traction}), (\ref{semi1}) and (\ref{semi2}).

We now take limits in equations (\ref{semi1}) and (\ref{semi2}) as $\etapar, \etaperp \to \infty$, 0, giving, 
\begin{eqnarray}
&& \phi_2(\bv_2-\bv_1)\cdot\bn=0, \quad \textrm {and}\quad (\bv_2-\bv_1)_{\|}=0, \label{impermeable-pure}\\
&&-[p]- \bn\cdot(\frac{\calT_2}{\phi_2})\bn=0, \quad \textrm {and}\quad  \calT_2\bn|_{\|}=0, \label{permeable-pure}
\end{eqnarray}
which correspond to the case of impermeable and fully permeable boundary, respectively.
  Rather than imposing the traction condition  (\ref{interface-traction}) on the whole interface $\partial\Omega_t$, we will consider mixed traction-displacement boundary
conditions. That is, 
\begin{eqnarray}
-P_0\bn=&&(\calT_1+\calT_2-pI)\bn, \,\,\, \textrm{on}\,\, \Gamma,  \label{boundary-traction0}\\
\tilde\bu=&& \tilde\bU, \quad \textrm{on} \,\,  \Gamma_0, \label{boundary-displacement}
\end{eqnarray}
where $\Gamma_0\cup\Gamma=\partial\Omega_t$,  $\Gamma_0\cap\Gamma=\emptyset,$
and $ \tilde\bu=\by-\by_0$ denotes the displacement vector, with $\by_0$ representing an equilibrium deformation field to be chosen later, and $\tilde\bU$ a prescribed boundary displacement. 


%

\medskip 

\noindent{\bf Remark.\,} The condition of semipermeability states the continuity of the force acting on the fluid across the interface. 
A related expression, with $\phi_2=1$ in  equation  (\ref{semi1}), is usually found in the literature.  


\subsection{Energy, dissipation and constitutive equations}
  The  total energy of system consisting of  the gel immersed in the environmental  fluid  is
\begin{eqnarray}
\mathcal F=&&
\int_{\Omega_t}(\sum_{i=1,2}\frac{1}{2}\gamma_i\phi_i|\bv_i|^2+  \psi(F, \phi_1, \phi_2))\,d\by \label{total-energy}\\
&& \psi:=\phi_1 W(F)+ G(\phi_1, \phi_2)
 \end{eqnarray}
 where $\psi$ denotes the free energy per unit deformed volume  of the gel, and   $W=W(F)$ and $G=G(\phi_1, \phi_2)$ represent the elastic energy density of the polymer and the Flory-Huggins energy of the gel, respectively, with 
 \begin{eqnarray} G(\phi_1, \phi_2)=&&a\, \phi_1\ln{{\phi_1}} + b\,\phi_2\ln{{\phi_2}} +c\, \phi_1\phi_2, \label{FH}\\
 a=&&\frac{K\theta }{V_mN_1},b=\frac{K\theta}{V_mN_2}, \,\, c=\frac{K \theta}{2V_m}\chi. \label{FH-coefficients}
\end{eqnarray}
Here $K\theta$ denotes the macroscopic energy unit, $V_m>0 $  is the volume occupied by one monomer, $N_1$ and $ N_2$ represent the number of lattice sites occupied by the polymer and solvent,
 respectively, and $N_x$ the number of monomers between entanglement points;  $\theta>0 $ denotes the absolute temperature, and $\chi>0$ represents the Flory interaction parameter \cite{flory} and \cite{GAS}. 
\begin{figure*}
 \centerline{\includegraphics[width=3.5in]{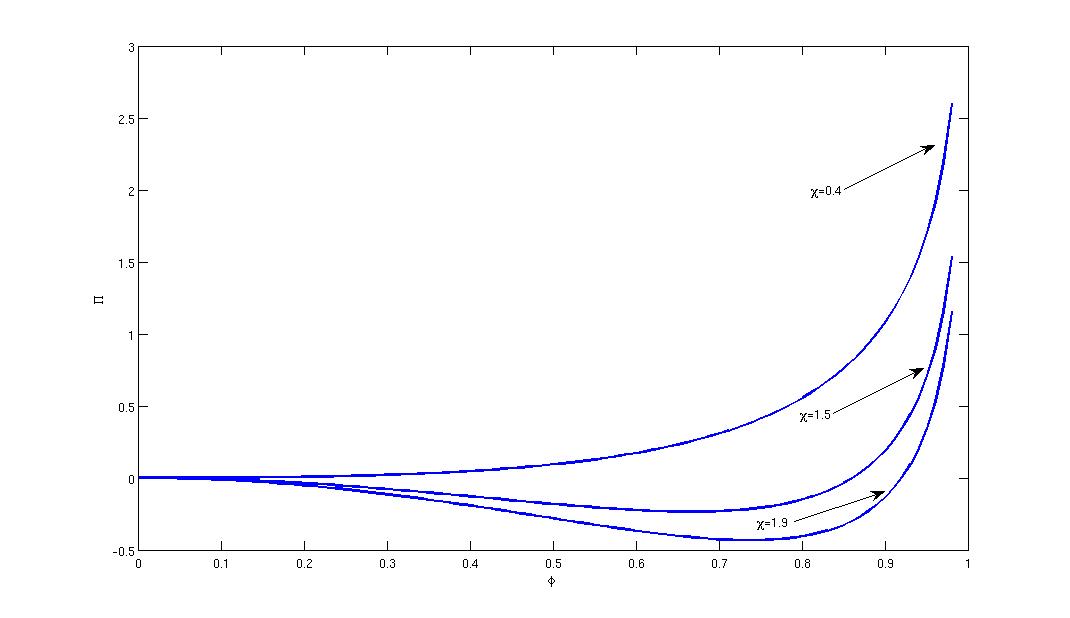}}
  \caption{{  These are plots of the osmotic pressure of the gel with respect to the polymer volume fraction $\phi$.  We point out the change of monotonicity of the graphs with respect to $\chi$. Such phenomenon corresponds to phase separation in the gel.   }} 
 \label{figure0}
 \end{figure*} 
 We let  $\calT_i^r$ and $\calT_i^v, i=1, 2$ denote the reversible and the viscous components of the gel stress, respectively. So,  the total stress of the $i$-component is 
$\calT_i= \calT^r_i+ \calT^v_i, \, i=1,2$, with $ \calT^r_2=0$. Moreover, we assume that the dissipative stress of each component is Newtonian, 
\begin{eqnarray}
\calT_i^v=&&\eta_i\mathbf D(\bv_i)+\mu_i(\nabla\cdot\bv_i)I, \,\, \calT_b=\frac{\eta_2}{2}D(\bv_b)\,\,\label{viscous-stress}\\
&&\quad \mathbf D(\bv):=\frac{1}{2}(\nabla\bv+\nabla\bv^T), \nonumber
\end{eqnarray} 
$i=1, 2$;  $\eta_i>0$ and $\mu_i>0$ denote the shear and bulk viscosity coefficients, respectively, of the components.
  The total  stress $\calT$ of the gel is 
\begin{equation}
\calT=\calT_1+\calT_2. \label{total-stress}
\end{equation}
From now on, we will treat (\ref{div}) as a constraint, and  let $p$ denote the corresponding Lagrange multiplier. Letting $\epsilon(\bx,t)$ denote the internal energy density of
 the gel and  $\zeta(\bx,t)=\frac{1}{\theta}(\epsilon(\bx,t)-\psi(\bx,t))$ the entropy density, the Clausius-Duhem inequality states that $\dot\zeta\geq 0$ holds, for all admissible processes of the gel.  
This allows us to establish the following proposition, which proof is presented in \cite{calderer-chabaud-zhang10}.
\begin{proposition}
Suppose that  the Clausius-Duhem inequality holds for all admissible processes,  and let  $\{\bv_i, \phi_i, p\}$  smooth solutions of  equations (\ref{sat}),
 (\ref{lin-momentum}),
 (\ref{mass-lagrangian}), (\ref{div}), (\ref{piperp}), (\ref{pipar}) and (\ref{viscous-stress}). Suppose that the boundary conditions (\ref{balance-mass-fluid-interface})-(\ref{pipar}) are satisfied. 
  Then the following relations hold:
 \begin{eqnarray}
 &&\mathcal T^r_1=\hat\sigma-\pi I,   \,\, \hat\sigma= \phi_1\frac{\partial W}{\partial F}F^T \label{reversible1,2}\\
 &&\pi= \phi_1(\frac{\partial G}{\partial \phi_1}-\frac{\partial G}{\partial \phi_2})-G.
\label{osmotic}
\end{eqnarray} 
Moreover, the dissipation inequality 
\begin{eqnarray}
\frac{d\mathcal F}{dt}=&&-\int_{\Omega_t}\bigg(\eta_i\norm \bD(\bv_i)\norm^2+\mu_i(\nabla\cdot\bv_i)^2 +\eta\norm\bv_1-\bv_2\norm^2\bigg)\,d\by
-\int_{\partial\Omega_t}(\etaperp w^2+\etapar |\bq|^2)\,dS,
      \nonumber\\
        \end{eqnarray}
holds, where  $\bn$ denotes the unit outward normal to the boundary.
\end{proposition}

As a result of the required  material frame-indifference, there exists a function $\hat W$ defined on the space of symmetric, positive definite tensors,  such that   $W(F)=\hat W(C), $  where $C=F^TF$.
 We let  $\mathcal P= 2\phi_I\frac{\partial \hat W(C)}{\partial C}$ denote the second Piola-Kirchhoff stress tensor.
We further  assume that the polymer is an isotropic elastic material, so there is  a scalar function $w $ of the principal invariants $\mathcal I=\{I_1, I_2, I_3\}$ of $C$,  such that 
$\hat W(C)=w(\mathcal I)$. In this case, the following representation holds (\cite{Gurtin-Fried10}, page 279):
\begin{eqnarray}
&&\frac{\partial\hat W}{\partial C}= \alpha_1(\mathcal I) I + \alpha_2(\mathcal I) C+ \alpha_0(\mathcal I) C^{-1}, \label{pk2-rep1}\\
&&\alpha_0= I_3\frac{\partial  w}{\partial I_3},\quad \alpha_1=\frac{\partial  w}{\partial I_1}+I_1\frac{\partial w}{\partial I_2}, \quad \alpha_2=-\frac{\partial w}{\partial I_2}.\label{alpha}
\end{eqnarray}
The Cauchy stress tensor  $\hat \sigma=\frac{1}{\det F}F\mathcal P(C)F^T$  in (\ref{reversible1,2}) has the form,
\begin{eqnarray}
&& \hat\sigma=\phI\big( \beta_0({\mathcal I})I +  \beta_1({\mathcal I}) B+ \beta_2({\mathcal I})B^{-1}\big), \label{sigma-isotropic1}\\
&&\beta_0= \frac{2}{\sqrt {I}_3}(I_2\frac{\partial  w}{\partial I_2}+ I_3 \frac{\partial  w}{\partial I_3}), \,\,  \beta_1= \frac{2}{\sqrt I_3}\frac{\partial  w}{\partial I_1}, \,\, \beta_2=-2\sqrt{I_3}\frac{\partial w}{\partial I_2}.\label{beta}
\end{eqnarray}
\begin{example}
Let us consider a Hadamard material  (\cite{CI87}, \cite{rognes-calderer-micek09}), 
\begin{eqnarray} &&w(I_1, I_2, I_3)= \frac{\mue}{2}\big(\frac{a_1}{s}I_1^s + \frac{\alpha}{r}I_3^{-r} + \frac{a_3}{q}I_3^q\big),\quad \mu_E=K\theta\mu_x, \label{example}\\
&&\hat \sigma= \frac{\phi_I\mue}{\sqrt I_3}\big(\nu B-\kappa I\big),\label {pk2-relax} \\
&&\nu=\nu(\mathcal I)= a_1 I_1^{s-1}, \quad \kappa =\kappa(\mathcal I)= \frac{1}{\mue\phI}{\pi\sqrt{I_3}}+\alpha I_3^{-r}- a_3 I_3^q,  \label{kappa0}\\
&&\beta_1=\frac{\phi_I\mue}{\sqrt I_3} \nu, \,\, \beta_2=0, \,\, \beta_0= -\frac{\phi_I\mue}{\sqrt I_3}\kappa, \label{gamma0}
\end{eqnarray}where $\alpha>0$, $a_1>0,  a_3>0, $ $q, s, r\geq 1$  are constant. The parameter $\mu_x$ represents the crosslink density of the network.
\end{example}

Stress free bulk equilibrium states $(\phi_0, F_0, p_0)$  are constant fields  
satisfying (\ref{mass-lagrangian}), and 
\begin{equation} \calT_1^r(F_0, \phi_0)-\phi_0 p_0I= 0, \quad  -(1-\phi_0)p_0= 0, \label{bulk-equilibrium} \end{equation}
with $\calT_1^r $  as in  (\ref{reversible1,2}).
 It is easy to check that the former reduce to 
\begin{eqnarray}&& \beta_1 f^2+ \beta_2 f^{-2}= \phi_I^{-1}{\pi(\phi)}-\beta_0, \,\, {\textrm{where}}\label{equilibrium-general}\\
&&B=f^2I,   \,\,   \phi=\phi_I f^{-\frac{3}{2}}, \,\, f>0. \label{dilation} \end{eqnarray}
For the energy in (\ref{example}), equations (\ref{bulk-equilibrium}) become $ \nu B=\kappa I,$
 with $\nu$ and $\kappa $ as in   (\ref{kappa0}). Equivalently,
\begin{equation}
\frac{\pi(\phi)}{\mue\phi}+\alpha (\frac{\phi}{\phI})^{2r}-a_3(\frac{\phI}{\phi})^{2q}=a_1 3^{s-1}(\frac{\phI}{\phi})^{\frac{2}{3}s}, \label{equilibrium-reduced}
\end{equation}
with $\pi(\phi)$ as in (\ref{FH}) and (\ref{osmotic}). We summarize the previous statements in the following:
\begin{proposition}Suppose that the gel is isotropic. Then the stress free bulk states are uniform dilations or compressions satisfying equation
 (\ref{equilibrium-general}) and (\ref{dilation}). In the case that the energy is given by (\ref{example}), there is a unique equilibrium state provided the material parameters satisfy
\begin{equation}
 \frac{\alpha}{\phI}>a_1 3^{s-1}\phI^{\frac{2s}{3}}+ a_3\phI^{2q}+\frac{1}{\mue}(c+b-a). \label{condition-equilibrium}
 \end{equation}
  Moreover, $\kappa(\phi_0)>0$, and $\phi_0>\phi^*>0$, where $\phi^*$ satisfies $\kappa(\phi^*)=0$. 
  \smallskip
\end{proposition} 
  
  \noindent
   {\bf Remarks.\,}  We point out that the assumptions on the coefficients of  $w$ and $G$ are sufficient to guarantee the existence of a global minimizer of the total energy under appropriately prescribed boundary conditions \cite{Ball77}, including those of  displacement-traction type.  Inequality (\ref{condition-equilibrium})  gives insights on material parameter ranges that guarantee existence of unique stress free equilibrium states. 
\begin{enumerate}
\item  Phase separation may occur for material parameters such that $\kappa(\phi)$ is non-monotonic. A necessary condition for the latter to occur is that 
 $c>0$ in (\ref{FH-coefficients}) be sufficiently large, and therefore it corresponds to the case that $\pi(\phi)$ has negative  intervals; in device applications, usually $c=.5$.  
This is illustrated in Figure 2.1.
 \item   Holding $c$ fixed,  small values of  $\alpha>0$ or  $\mue>0$ may also lead to phase separation. The latter correspond to  prescribing small shear and bulk moduli. Moreover, the equilibrium value $0<\phi_0<1$ also increases with respect to $\mue$
\end{enumerate}
We now let $\phi_0, F_0$ denote a solution of the  boundary value problem
\begin{equation}
\nabla\cdot \calT_1^r(F, \phi)=0, \quad \phi\det F=\phI,  \,\,\, p= \textrm{constant}, \label{general-equilibrium}
\end{equation}
subject to mixed displacement-traction boundary conditions.

\section{ Linear problems}
We now  linearize the governing  system  (\ref{lin-momentum}), (\ref{chainrule}), (\ref{mass-lagrangian}) and (\ref{div}) about  a particular  time independent solution $(\phi_0, F_0, p_0=0, \bv_1=\mathbf 0=\bv_2)$.
 Relevant special cases include   stress-free dilation or compression states $F_0=f_0 I$, and also non-stress free equilibria. 
The {\it tilde} notation represents perturbations from the equilibrium state,  which is labeled with {$ 0$}-super (or sub) indices. 
 Let $\tilde F:=\nablax\tilde\bu$, and 
\begin{equation}\label{perturbation}
\phi=\phi_0+\tilde\phi, \,\, F=F_0+\tilde F, \,\, p=\tilde p, \,\,
\bv_1= \tilde \bv, \,\, \bv_2=\tilde\bw.
\end{equation}
The gel domain now corresponds to the reference configuration, $\Omega$, of the polymer.
First, we calculate the linear swelling ratio and polymer volume fraction:
\begin{eqnarray}
&&\det F= \phi_0\det F_0(1+\gamma(\nablax\tilde\bu))+ o(|\nablax\tilde\bu|^2),\label{lin-det}\\
&&\phi=\phi_0 (1-\gamma(\nablax\tilde\bu))+o(|\nablax\tilde\bu|^2),\label{lin-phi} 
\end{eqnarray}
 where $ \gamma(\nablax\tilde\bu)=\tr(F_0^{-1}\nablax\tilde\bu)$.  Letting $C_0=F^T_0F_0$, we  denote
\begin{eqnarray}
\mathfrak C_{ijkl}=&&2\frac{\partial^2 \hat W}{\partial C_{ij}\partial C_{kl}}(C_0)=\frac{\partial \mathcal P_0}{\partial C}(C_0), \label{C-matrix}\\
\pi^0=&&\pi(\phi_0, 1-\phi_0), \quad\pi^0_{j}= \frac{\partial\pi}{\partial \phi_j}(\phi_0, 1-\phi_0), \,\,\, j=1,2. \label{pi-linear}
\end{eqnarray}
The fourth order tensor with components $\mathfrak C_{ijkl}$ corresponds to the elasticity matrix, with the symmetry properties 
\begin{equation}
\mathfrak C_{ijkl}=\mathfrak C_{klij}=\mathfrak C_{jikl}=\mathfrak C_{ijlk}. \label{symmetry-C}
\end{equation}
The quantities (\ref{lin-det})-(\ref{pi-linear}) yield the linearized expressions of the stress tensors. These equations  are exact up to terms of order $o(|\nablax\tilde\bu|^2)$):
\begin{eqnarray}
\hat\sigma=&&\hat\sigma_0+\phi_0\{-F_0\mathcal P_0F_0^T\gamma(\nablax\tilde\bu)+
 F_0\mathfrak C(F_0^T\nablax\tilde \bu+ \nablax\tilde\bu^T F_0)F^T_0\nonumber\\ &&+ \big(F_0\mathcal P_0\nablax\tilde\bu^T+ 
\nablax\tilde\bu \mathcal P_0 F_0^T\big)\}, \label{hatsigma-lin} \\
\pi=&& \pi^0 -\phi_0( \pi_{1}^0-\pi_{2}^0)\,\gamma(\nablax\tilde\bu), \nonumber \\
\calT=&&\calT_1^r=\calT_1^{r,0}+
 (\phi_0\big(\pi_{1}^0-\pi_{2}^0)I-\hat\sigma(F_0, \phi_0)\big)\gamma(\nablax\tilde\bu)\nonumber\\
 &&+\phi_0\{ F_0\mathfrak C(F_0^T\nablax\tilde \bu+ \nablax\tilde\bu^T F_0)F^T_0+ \big(F_0\mathcal P_0\nablax\tilde\bu^T+ 
\nablax\tilde\bu \mathcal  P_0 F_0^T\big)\}\label{linear-stress-general}
\end{eqnarray} 
The linearized system of equations is 
\begin{eqnarray}
&& \dive (\phi_0 \tilde\bv +(1-\phi_0)\tilde\bw)=0, \label{linearized-divergence} \\
&& \dive \mcT_1 -\beta (\tilde\bv-\tilde\bw)-\phi_0\nabla p=\mathbf 0, \label{linearized-momentum-1} \\
&& \dive \mcT_2 +\beta(\tilde\bv-\tilde\bw)-(1-\phi_0)\nabla p=\mathbf 0, \label{linearized-momentum-2} \\
&& \mcT_1(\nablax\tilde\bu,\nabla\tilde\bv) =\calT_1^r(\nablax\tilde\bu) +\eta_1 \bD(\tilde\bv) +\mu_1 (\dive\tilde\bv) I, \label{linearized-stress-1} \\
&& \mcT_2(\nablax\tilde\bu,\nabla\tilde\bw) = \eta_2 \bD(\tilde\bw) +\mu_2 (\dive\tilde\bw) I. \label{linearized-stress-2}\\
&&\tilde\bu_t=\tilde\bv, \label{ut-v}
\end{eqnarray}
together with (\ref{lin-phi}). We point out that the last equation follows from the linearization of (\ref{chainrule}), neglecting uniform translations:
\begin{equation}
\nablax\dot\bu=(\nabla\bv_1) F_0. \label{chainrule-linear1}
\end{equation}
\subsection{Stability of Equilibrium Solutions}The conditions that guarantee the stability of the equilibrium states turn out to be also necessary conditions for the solvability of the time-dependent, quasi-static problem, that we study in later sections.   In order to established such conditions, we first 
outline the properties of the second order tensor $\mathcal  P$ and that of the fourth order one $\mathfrak C$. Unlike the case  of linearizing about a stress free state,  here we need to include $\mathcal T_r^0$ in the analysis of stability.  

{\bf Notation.\,} With the understanding that the quantities that we study are evaluated at $F_0, C_0=F^T_0F_0$, in this section, we suppress the $0$-notation in the equilibrium solution, and the $\it tilde$ symbol  in the perturbation terms, unless explicitly needed. 
 The fourth order elasticity tensor
 \begin{eqnarray}
\mathfrak C_{ijpq}(C)
&& ={\scriptsize\mathcal A}_1\delta_{pq}\delta_{ij} + {\scriptsize\mathcal A}_2(\tr C\delta_{pq}-C_{pq}) C_{ij}+
{\scriptsize\mathcal A}_3 (\det C) C_{pq}^{-T} C_{ij}^{-1}\nonumber \\
 && + \alpha_2 \delta_{ip}\delta_{jq} -\alpha_0 C^{-1}_{ip} C^{-1}_{qj}, \,\, {\scriptsize{\mathcal A}}_m= \sum_{n=0}^2\frac{\partial\alpha_n}{\partial I_m}, \,\,\, m=1, 2, 3. \label{elasticity-general}
 \end{eqnarray}
 In the case that $C$ corresponds to a pure expansion or compression, $C=f^2I$, $f>0$, we obtain the following representations.
\begin{eqnarray}
\mathfrak C_{ijpq}=&& \lambda(f) \delta_{ij}\delta_{pq}+ 2\mu(f)\delta_{ip}\delta_{jq}, \quad \textrm{with} \label{matrixC}\\
\lambda=&&\sum_{n=0}^2\big(\frac{\partial \alpha_n}{\partial I_1}+ \frac{\partial \alpha_n}{\partial I_2}f^4+\frac{\partial \alpha_n}{\partial I_3}f^2\big),
 \quad \mu=\alpha_2-\alpha_0 f^{-4}. \label {lambda-mu-isotropic}
\end{eqnarray} 
Moreover, the total linearized stress tensor (\ref{linear-stress-general}) becomes
\begin{eqnarray}
&&\calT^r=  2\tilde\mu \mathbf D(\bu) + \tilde\lambda \tr{\mathbf D(\bu)} I , \quad {\textrm {with}} \\
&&\tilde \mu =\phI\big(2\mu+ f_0^{-2}(\alpha_1+\alpha_2f_0^2+\alpha_0 f_0^{-2}\big),\label{tilde-mu}  \\
&&\tilde\lambda=  \phi_0(\pi_1^0-\pi_2^0)-\pi^0+ 2\phI \lambda. \label{tilde-lambda}
\end{eqnarray}
Likewise, the analog of the fourth order tensor (\ref{elasticity-general}) that combines the elastic and Flory-Huggins effects is
\begin{equation}{\tilde{\mathfrak C}}_{ijpq}= \tilde\lambda(f) \delta_{ij}\delta_{pq}+ 2\tilde\mu(f)\delta_{ip}\delta_{jq}.\label{matrixC-special}
\end{equation}
\begin{proposition}
Let $f>0$.  Suppose that 
\begin{equation}
\tilde\mu(f)> 0, \quad 3 \tilde\lambda(f) + 2\tilde\mu(f) >0, \label{stability}
\end{equation}
hold. Then  $\tilde{\mathfrak C}$  in (\ref{matrixC-special}) is coercive, that is, there exists a constant $\mu_0>0$ such that
 \begin{equation}{\tilde{\mathfrak C}}_{ikls} A_{ik}A_{ls}\geq \mu_0 |A|^2 \label{coercivity}  \end{equation} holds,  for all $A\in M^{3\times 3}$. 
\end{proposition}
For the Hadamard energy in (\ref{example}),  with $\kappa$ and $\nu$  as  in (\ref{kappa0}),  we have
\begin{eqnarray}
&&\tilde \mu =\phI(f^{-2}\kappa+\nu)f^{-2},  \quad 
\tilde\lambda=  \phi(\pi_1^0-\pi_2^0)-\pi^0+ 2\phI \lambda. \label{tilde-lambda-mu}
\end{eqnarray}

\subsubsection{General equilibrium state}
 We assume that $(F_0, \phi_0) $ is a solution of (\ref{general-equilibrium}), and let  $\mathcal T^r$ be as in (\ref{linear-stress-general}), with the elasticity tensor $\mathfrak C$ given by (\ref{elasticity-general}).  
 We say that an equilibrium solution is locally stable if 
 \begin{equation}
 \calT^r(F_0, \phi_0)\cdot\nabla\bv\geq 0,
 \end{equation} 
 for all sufficiently smooth fields $\bu, \bv$ satisfying (\ref{chainrule-linear1}).   We now derive sufficient conditions for  the stability of equilibrium solutions.  
Let us introduce the following notation.
\begin{eqnarray}
&& C_2:=\frac{\partial^2 w}{\partial I_1^2} +\frac{\partial^2w}{\partial I_1\partial I_3}I_3, \quad 
 C_3=I_3^2\frac{\partial^2 w}{\partial I_3^2} +\frac{\phi_1}{2}\pi_{1,2}. \label{C}
\end{eqnarray}
\begin{proposition} 
Suppose that $w=w(I_1, I_3)$ and that relations (\ref{pk2-rep1})-(\ref{beta}) and  (\ref{chainrule-linear1}) hold.  
 Let   $0<\phi_0(\bx)<1$ and $F_0(\bx)\in \mathcal M^{n\times n}_+, \, \bx\in\Omega$,  be an  equilibrium solution. 
 Then
 \begin{equation}
 \calT^r\cdot\nabla\bv= \calT^r_0\cdot\nabla\bv+   \phI\frac{\partial}{\partial t} \mathcal H,  \label{total-time-derivative}
 \end{equation}
where
 \begin{eqnarray}
 && \mathcal H:=  \mathfrak C(\nablax\bu^T F_0)\cdot(\nablax\bu^TF_0)^T +\mathfrak D(\nablax\bu; F_0)
   +\frac{1}{2}\pi_{1,2}^0\,\tr^2(F_0^{-1}\nablax\bu), \label{calH}\\
&&\mathfrak C(\nablax\bu^TF_0)\cdot(\nablax\bu^TF_0)=C_2\,\tr^2 (\nablax\bu^T F_0)+ I_3^2\frac{\partial^2 w}{\partial I_3^2}\,\tr^2(F^{-1}_0\nablax\bu) -I_3\frac{\partial w}{\partial I_3}|(\nablax\bu) F^{-1}_0|^2, \label{C-general}\\
&&\mathfrak D(\nablax\bu; F_0):= \frac{\alpha_0}{2}\tr\big((F^{-T}_0\nablax\bu^T)^2+ \nablax\bu\, C^{-1}_0\nablax\bu^T\big)+\alpha_1\big(F_0\nablax\bu^T\cdot\nablax\bu \,F^{-1}_0
  + |\nablax\bu|^2), \label {D}\\
  &&\pi^0_{1,2}(\phi_1):=\pi_1^0-\pi_2^0= a-b-2c\phi_1+\frac{b}{1-\phi_1}. \label{pi1-pi2}
\end{eqnarray}
 \end{proposition}
\begin{proof}
Starting with 
\begin{eqnarray}
\mathcal T^r\cdot \nabla\bv =&&\bigg(\calT_r^0 +\phI\{F_0\mathfrak C(F_0^T\nablax\tilde \bu+ \nablax\tilde\bu^T F_0)F^T_0\nonumber\\+&& \big(F_0\mathcal P_0\nablax\tilde\bu^T+ 
\nablax\tilde\bu \mathcal  P_0 F_0^T\big)\}  + \phI(\pi_1^0-\pi_2^0)I)\gamma(F_0^{-1}\nablax\bu)\bigg)\cdot\nabla\bv, \label{energy-law-effective}
\end{eqnarray}  straightforward calculations that apply (\ref{chainrule-linear1})  give
\begin{eqnarray}
&&\frac{\partial}{\partial t}\big(\mathfrak C(\nablax\bu^TF_0)\cdot(\nablax\bu^TF_0)\big)  = \mathfrak C(\nablax\bu^TF_0+ F_0^T\nablax\bu)\cdot(\nablax\dot\bu^TF_0), \\
&&\frac{\partial\mathfrak D}{\partial t}=(F_0\calP_0\nablax\bu^T+\nablax\bu\calP_0 F^T_0)\cdot\nabla\bv, \label{frakD}\\ 
&&\frac{\partial}{\partial t}(\frac{\phI}{2}\pi^0_{1,2}\, \tr^2(F^{-1}_0\nablax\bu))=\phI\pi^0_{1,2} \tr(F^{-1}_0\nablax\bu)I\cdot(\nablax\dot\bu F^{-1}_0),\end{eqnarray}
where $\calP$ is given in (\ref{pk2-rep1}).
\end{proof}

Next, we establish coercivity of the operator $\mathcal H$ in (\ref{calH}). For this, let us write
\begin{eqnarray}
&& \mathcal H:= \mathcal H_1 + \mathcal H_2, \\
&& \mathcal H_1:= C_2 \tr^2(\nabla\bu^T F_0)+ C_3 \tr^2(F^{-1}_0\nabla\bu)- \alpha_0|\nabla\bu F^{-1}_0|^2+ \frac{\alpha_0}{2}\tr(\nabla\bu C_0^{-1}\nabla\bu^T), \label{H1} \\
&& \mathcal H_2:= \frac{\alpha_0}{2}\tr(F^{-T}_0\nabla\bu^T)^2 +\alpha_1(F_0\nabla\bu^T\cdot\nabla\bu F^{-1}_0 +|\nabla\bu|^2).
\end{eqnarray}
 \begin{lemma}
Let $a, b$, $c$ be as in (\ref{FH}) and $0<\chi\leq 0.5$. Then  $\pi^0_{1,2}>0$ is  monotonically increasing. Moreover if for each $I_1>0$,  $w(I_1, I_3)$ is convex with respect to $I_3$, then $C_3>0$.
\end{lemma}

This condition on $\chi$ is satisfied in gels used in device applications. The  monotonicity of $\pi^0_{1,2}$ for this range of   $\chi$ is illustrated in {\it Figure\,2.1.}
\smallskip

\begin{proposition}
Suppose that the assumptions of Lemma 3.3 hold. Assume that $\alpha_0<0$ and that $C_2>0$.  Then 
\begin{equation}
\mathcal H_1\geq |\alpha_0||\nabla\bu F^{-1}_0|^2 +C_2 \tr^2(\nabla\bu^TF_0) +C_3\tr^2(F^{-1}_0\nabla\bu). \label{coercivityH1}
\end{equation}
\end{proposition}
Next, we study the coercivity of $\mathcal H$.  
Let us  consider the polar decomposition  $F=RV$, where $R$ denotes the rotation tensor, and $V=\sqrt{C}$. Let $\lambda_i, \, i=1, 2, 3$ denote the eigenvalues of $C$.  Let us denote
\begin{eqnarray}
N:=&&(\nablax\bu^T) R, \\
\Gamma_{ij}:=&& h_{ij}N_{ij}N_{ji}, \quad h_{ij}:=\frac{\alpha_0+\alpha_1(\lambda_i+\lambda_j)}{\sqrt{\lambda_i\lambda_j}}, \quad i\neq j.
\end{eqnarray}
We now calculate $\mathcal H_2$  using its representation in terms of  the eigenvector basis of $C$, 
\begin{eqnarray}
\mathcal H_2=&& \frac{\alpha_0}{2}\sum_I\lambda_i^{-1} N_{ii}^2+ \frac{\alpha_0}{2}\sum_{i\neq j}\Gamma_{ij} +2\alpha_1\sum_{i}N_{ii}^2 +\alpha_1\sum_{i\neq j}N_{ij}^2
\end{eqnarray}
We combine the first term of the right hand side of $\mathcal H_2$ with  the last one  on the right hand side of $\mathcal H_1$ (\ref{H1}) (which can also be written in 
terms of $N_{ij}$).  We also combine the mixed products in $\Gamma_{ij}$   with  the last term in $\mathcal H_2$, upon application of the Cauchy-Schwartz inequality. 
We now state
\begin{theorem} Let $(F_0, \phi_0)$ be an equilibrium solution.  Suppose that the assumptions of Proposition 3.4 hold. Furthermore, we assume that $C_2>0$ in (\ref{C})  and 
\begin{equation}
\alpha_1>\frac{1}{2}\max_{i\neq j}|h_{ij}|.
 \label{H2-coercivity}
\end{equation}
Then 
\begin{equation}
\mathcal H\geq \frac{|\alpha_0|}{2}|\nabla\bu F^{-1}_0|^2 +C_2 \tr^2(\nabla\bu^TF_0) +C_3\tr^2(F^{-1}_0\nabla\bu). \label{coercivityH}
\end{equation}

\end{theorem}
\noindent
{\bf Remarks.\,}

\noindent
{\bf 1.}\, Inequalities (\ref{stability}) and (\ref{coercivity}) (for   a spherical equilibrium state),  and the positivity of $C_2$ and $C_3$ in (\ref{C})
 (for an arbitrary equilibrium state)  correspond to the strong ellipticity of the linear operator. Strong ellipticity  guarantees regularity of the weak 
solutions of the linear problem. 
In the case that $F_0=f_0 I$, the assumptions of Theorem 3.5 imply  inequalities  (\ref{stability}) to hold. 
\smallskip

\noindent
{\bf 2.}\, The need to separately account for stretch and rotation in the proof of Theorem 3.5  is a signature feature of linear  elasticity,
 when the equilibrium state is not stress free. In particular,  the theorem applies to the linearization about the reference configuration, 
even if the residual stress is nonzero. In this case, inequality (\ref{H2-coercivity}) is identically satisfied.
\subsection{Initial, boundary-value problems} We formulate the governing equations in terms of homogeneous boundary conditions on the displacement field, which also satisfies  the only initial condition to be  specified in the problem,
\begin{eqnarray}&&\but|_{t=0} = \but_0,  \quad  \but_0|_{\Gamma_0} = \tilde\bU|_{t=0}.\label{init-condition}
\end{eqnarray}
The latter is a compatibility condition with the boundary data at $t=0$. 
  Assume that $\Gamma_0$ is of class $ C^m$, for some given integer $m \geq 1$, and $\tilde\bU \in H^{m-1/2}(\Gamma_0)$.   We let 
$\uu $ denote the extension of $\tilde\bU$ to $\Omega$, so that  $||\uu||_{H^m(\Om)} \leq C||\tilde\bU||_{H^{m-1/2}(\Gamma_0)}$ \cite[p. 68]{La69}. From now on, we will set $m=2$.
We also assume $P_0 \in H^{1/2}(\pOm)$ and $\pOm \in C^1$.  Then  $\exists P \in H^1(\Om)$ such that $P|_{\pOm} = P_0$ and $||P||_{H^1(\Om)} \leq C||P_0||_{H^{1/2}(\pOm)}$. 

Let,
\begin{equation}
\bar \bu=\but-\uu, \quad \bar\bu_0 =\but_0-\bU_{t=0},  \quad   \bar\bv_2=\tilde\bv_2-\bV\quad \textrm{and} \quad \bar p=\tilde p- P, \quad \bar\bu_t=\bv_1, \label{u-v-p}
\end{equation}
where $\bV$ and $P$ are defined as follows:
\begin{eqnarray}
&&\bV=\bU_t, \, \, P=0 \,\,\\
&&\bV=0, \,\, P\, \textrm{ is the extension of}\,  P_0,
\end{eqnarray}
for impermeable and fully permeable boundary, respectively. The governing system reduces now to 
\begin{eqnarray}
&& \dive ({\phi_0} \bar\bu_t +(1-{\phi_0}) \bar\bv_2) = h, \label{zero-divergence} \\
&& \dive \mcT_1(\nablax\bar\bu,\nabla\bar\bv_1)-\phi_0\nabla\bar p -\beta(\bar\bu_t-\bar\bv_2)=\mathbf f_1, \label{zero-momentum-1} \\
&& \dive \mcT_2(\nablax\bar\bu, \nabla\bar\bv_2)-(1-\phi_0)\nabla \bar p +\beta(\bar\bu_t-\bar\bv_2)=\mathbf f_2, \label{zero-momentum-2} \end{eqnarray}
with $\calT_1$ and $\calT_2$ as in (\ref{linearized-stress-1}) and (\ref{linearized-stress-2}), respectively, and 
\begin{eqnarray} && h= -\dive ({\phi_0}\uu_t +(1-{\phi_0})\vv), \,\, \kappa= \frac{1}{\beta}(1-{\phi_0})^2, \label{h-kappa}\\
&& \mathbf f_1=-\nabla\cdot\mathcal T_1(\bD(\uu),\bD(\uu_t))+\phi_0\nabla P+ \beta(\uu_t-\vv)\label{f1} \\
&&\mathbf f_2=-\nabla\cdot\mathcal T_2(\bD(\vv))+ (1-\phi_0)\nabla P- \beta(\uu_t-\vv) \label{f2}\\
&&\mathbf G=(P -P_0)\bn- \mathcal T(\mathbf D(\bU), \mathbf D(\bU_t), \mathbf D(\vv))\bn,\,\, \textrm {on}\, \partial\Omega,\label{G}\\
&& H=h+\nabla\cdot(\frac{1-\phi_0}{\beta})\mathbf f_2 := -\nabla\cdot \mathbf{\mathcal H}, \label{Hh}
\end{eqnarray}

{\bf Notation.\,} 
We suppress the superimposed {\it bar} on the unknown fields, and write $(\bu, \bv_2, p)$. 
\medskip

Without loss of generality, in sections 4 and 5, we consider linearization of the original system about uniform expansion or compression, $F_0=f_0 I$.
With the help of the energy law developed in section 6, these results can be easily extended to the case  of  a general equilibrium state.

\section{Inviscid solvent}
We prove existence and uniqueness of  weak solution in the case that the fluid component is inviscid. Setting
  $\eta_2=0$ and $\mu_2=0$ in (\ref{linearized-stress-2}) and solving it explicitly for  $\bv_2$, yields the governing system:
\begin{eqnarray}&&\bv_2 =\bv_1 -\frac{1}{\beta}\big(\nabla p+\mathbf f_2\big),\label{fluid-velocity}\\                                                                    
&& \dive \big(\bv_1 -\kappa \nabla p\big)=H, \label{inviscid-divergence} \\
&& \dive \mcT = \mathbf f_1+\mathbf f_2, \label{inviscid-momentum}\\
&& \mathcal T = \mathcal T^r+\frac{\eta_1}{2} (\nabla\bv_1 +\nabla \bv_1^T)+\mu_1(\dive\bv_1)I,\end{eqnarray}
with $\mathcal T^r$ as in (\ref{linear-stress-general}).
 We will analyze two cases that correspond to impermeability and full permeability of the boundary, respectively. 
\subsection{Impermeable boundary}
 We assume that the boundary of the gel is impermeable to solvent, so that the normal component of the vectorial condition (\ref{impermeable-pure})  holds on $\partial\Omega$. This combined with 
 equation (\ref{fluid-velocity})  reduces to requiring 
\begin{equation} \nabla  p \cdot \bn|_{\pOm} =0 \label{nabla-p} \end{equation}
on the  pressure.    
Moreover, following Feng and He \cite{fenghe}, we define the variable
$ q = \dive \bu,$
 which measures the volume change of the solid network of the gel. 
 The system of equations can be reformulated as
\begin{eqnarray}
&& q=\dive \bu, \label{imperm-qeqn-1} \\
&& q_t - \nabla\cdot(\kappa\nabla p) =H
 \label{imperm-qeqn-2} \\
&&\nabla\cdot \mathcal T(\nablax\bu, \nabla\bu_t,  q_t)-\nabla p = \mathbf f_1+ \mathbf f_2,\label{imperm-qeqn-3}\\
&& \mcT(\nablax\bu,\nabla\bu_t,  q_t) =\tilde\lambda\tr(\mathbf D(\bu))I+ \tilde \mu \mathbf D(\bu)+\eta_1 \bD(\bu_t) +\mu_1 q_t I, \label{stress-imperm}
\end{eqnarray}
with  $\tilde\lambda$ and $\tilde\mu$ as in (\ref{tilde-lambda-mu}). The quantities $H, \mathbf f_1$, $\mathbf f_2$ and $\mathbf G$ are  as in  (\ref{h-kappa})-(\ref{Hh})  with  $P=0$.
  The initial and boundary conditions on solutions of this  system are
\begin{eqnarray}
&&\bu|_{t=0} = \bu_0,\,\, 
q|_{t=0} =\dive \bu_0, \label{u0}\\
&&\bu|_{\Gamma_0} = \mathbf 0, \bu_0|_{\Gamma_0} = \mathbf 0, \label{u-gamma}\\
&& \mathcal T(\nablax\bu, \nablax \bv_1,  p, q_t)\bn|_{\Gamma}= \mathbf G,\label{boundary-stress-linear}
\end{eqnarray}
together with (\ref{nabla-p}). In order to prove existence  of weak solution of the governing system, we first derive an energy law. For this,  we multiply (\ref{imperm-qeqn-3}) by $\bu_t$ and use equations (\ref{imperm-qeqn-1}) and (\ref{imperm-qeqn-2})  and integrate by parts over $\Om$,  applying  the boundary conditions,
\begin{eqnarray}
&& \frac{d}{dt} \int_{\Om}\frac{1}{2} \big({\tilde \mu}|\bD(\bu)|^2 +{\tilde \lambda}  q^2\big)\,d\bx +\int_{\Om}\frac{1}{2} \big(\kappa |\nabla p|^2+ {\eta_1} |\bD(\bu_t)|^2+\mu_1 q_t^2\big)d\,\bx \nonumber\\ \,\,\,\, && = \int_{\Gamma} \mathbf G \cdot \bu_t + \int_{\Omega}(\mathbf f_1+\mathbf f_2)\cdot \bu_t\,d\bx.
\label{energy-law-1}
\end{eqnarray}
We  introduce the   function spaces and  notational conventions:
 \begin{eqnarray*}
 &&\bW(\Omega)=  \{ {\boldsymbol{\omega}}\in H^1(\Om): \boldsymbol{\omega} |_{\Gamma_0} = \mathbf 0 \},\,\,\, 
H_0^1= \{ \boldsymbol\omega \in H^1(\Om): \boldsymbol\omega |_{\partial\Omega}=0\},\\
 &&  \Phi = L^2(\Om), \quad  \Psi = H^1(\Om), \\
 &&\Phi_0= \{ q \in L^2(0,T; L^2(\Omega)): \int_0^T \int_{\Omega} q\,d\bx\,dt=0 \}, \\
&&\bcalw_1=\{ \om \in L^2(0,T;\bW(\Omega)): \int_0^T \int_{\Omega} q_0 \nabla \cdot \om \,d\bx\,dt= 0, \, \forall\, q_0 \in L^2(0;T;\Phi_0) \}, \\
&&\bcalw_2=\{ \om \in L^2(0,T;\bW(\Omega)): \int_0^T \int_{\Omega} c \nabla \cdot \om \,d\bx\,dt= 0, \, \forall \, c \in \mathbf R \}.
 \end{eqnarray*}
   We write $\boldsymbol\omega\in H^1_0$ to indicate that every component of the vector function $\boldsymbol\omega$ is a scalar function  in  $H^1(\Omega)$ vanishing on the boundary. 
\begin{definition}
 $(\bu,q,p) \in \mathbf \bW \times L^2(\Om) \times H^1(\Om)$ is a weak solution if $\forall$ $(\boldsymbol\omega, \varphi, \psi) \in  \bW \times L^2(\Om) \times H^1(\Om)$,
\begin{eqnarray}
&& \int_{\Om} \varphi q \,d\bx= \int_{\Om} \varphi \dive \bu\,d\bx, \label{imperm-weak-1} \\
&& \int_{\Om} q_t \psi\,d\bx +\int_{\Om} \kappa \nabla p \cdot \nabla \psi\,d\bx = \int_{\Om}H \psi\,d\bx, \label{imperm-weak-2} \\
&& \int_{\Om}\{ \big(- p +\tilde\lambda q+\mu_1q_t)I+ \tilde \mu \mathbf D(\bu)+\eta_1 \bD(\bu_t)\}\cdot \mathcal D(\boldsymbol\omega)\,d\bx \nonumber \\
&& = \int_{\Gamma} \mathbf G\cdot \boldsymbol\omega\,dS+\int_{\Omega}(\ff_1+\ff_2)\cdot\boldsymbol\omega\,d\bx,  \label{imperm-weak-3}\\
&& \int_{\Om} \bD(\bu(0))_{ij} \bD(\boldsymbol\omega)_{ij}\,d\bx = \int_{\Om} \bD(\bu_0)_{ij} \bD(\boldsymbol\omega)_{ij}\,d\bx, \,\,\, \int_{\Om} q(0) \varphi\,d\bx = \int_{\Om} q_0 \varphi\,d\bx. \label{imperm-weak-5} 
\end{eqnarray}
\end{definition}
Since no boundary conditions are prescribed on $p$, an inf-sup condition is needed to establish compactness. 
\begin{lemma}
For any positive $T< \infty$, there exists $\alpha_0>0$ such that 
\begin{equation}
\sup_{\bw \in L^2(0,T;\bW(\Omega))} \frac{|\int_0^T \int_{\Omega} \phi \nabla \cdot \bw\,d\bx\,dt|}{||\bD(\bw)||_{L^2(0,T;L^2(\Omega))}} \geq \alpha_0 ||\phi||_{L^2(0,T;L^2(\Omega))}, \quad \forall \phi \in L^2(0,T;L^2(\Omega)). \label{inf-sup1}
\end{equation}
\end{lemma}
\textit{Proof:} 
The proof presented here is due to Sayas \cite{sayas1} but is a special case of the general LBB condition \cite{Br74}. 
 Since $L^2(0,T; L^2(\Omega)) = \Phi_0 \oplus \mathbf R$  (that is, $\Phi_0$ is the orthogonal complement of $\mathbf R$ under the $L^2$ inner product),  it is clear that the inequality (\ref{inf-sup1}) is equivalent to 
\begin{equation}
\sup_{\bw \in L^2(0,T;\bW(\Omega))} \frac{|\int_0^T \int_{\Omega} (\phi_0 +c)\nabla \cdot \bw\,d\bx\,dt|}{||\bD(\bw)||_{L^2(0,T;L^2(\Omega))}} \geq \alpha_0 [||\phi_0||_{L^2(0,T;L^2(\Omega))} +|c|] \quad \forall \phi_0 \in \Phi_0, \ \forall c \in \mathbf R. \label{inf-sup1-equiv}
\end{equation}
By \cite{gatica1}, (\ref{inf-sup1-equiv}) holds if and only if the following are valid:
\begin{enumerate}
\item There exists an $\alpha_1>0$ such that 
\begin{equation}
\sup_{\bw \in L^2(0,T;\bW(\Omega))} \frac{|\int_0^T \int_{\Omega} \phi_0 \nabla \cdot \bw\,d\bx\,dt|}{||\bD(\bw)||_{L^2(0,T;L^2(\Omega))}} \geq \alpha_1 ||\phi_0||_{L^2(0,T;L^2(\Omega))}, \quad \forall \phi_0 \in\Phi_0, \nonumber
\end{equation}
\item There exists an $\alpha_2>0$ such that 
\begin{equation}
\sup_{\bw \in L^2(0,T;\bW(\Omega))} \frac{|\int_0^T \int_{\Omega} c \nabla \cdot \bw\,d\bx\,dt|}{||\bD(\bw)||_{L^2(0,T;L^2(\Omega))}} \geq \alpha_2 |c|, \quad \forall c \in \mathbf R, \nonumber
\end{equation}
\item $L^2(0,T;\bW(\Omega))=\bcalw_1 + \bcalw_2.$
\end{enumerate}
Note that since $L^2(0,T;H_0^1(\Omega)) \subseteq L^2(0,T;\bW(\Omega)),$ the first item holds if 
\begin{equation}
\sup_{\bw \in L^2(0,T;H_0^1(\Omega))} \frac{|\int_0^T \int_{\Omega} \phi_0 \nabla \cdot \bw\,d\bx\,dt|}{||\bD(\bw)||_{L^2(0,T;L^2(\Omega))}} \geq \alpha_1 ||\phi_0||_{L^2(0,T;L^2(\Omega))}, \,\,\, \forall \phi_0 \in L^2(0,T;\Phi_0). \nonumber
\end{equation}
The latter is  a well-known result shown in \cite{GiRa79}. In order to prove the validity of the second item, note that for $\bw \in L^2(0,T;\bW(\Omega)),$ $\int_0^T \int_{\Omega} c \nabla \cdot \bw\,d\bx\,dt = \int_0^T \int_{\Gamma} c \bw \cdot \bn\,dS\,dt.$  Thus the result holds if it is possible to find a $\bw \in L^2(0,T;\bW(\Omega))$ satisfying $\int_0^T \int_{\Gamma} \bw \cdot \bn \neq 0,$  where $\bn$ denotes the unit outward normal to $\Gamma$.  Assuming that $\Gamma$ is  Lipschitz, we take $\tilde{\Gamma}_1 \subset \Gamma$ with nonzero measure, and a fixed vector $\mathbf m \in \mathbf R^n$ such that for some $\delta_0>0,$
\[ \mathbf m \cdot \bn(\by) \geq \delta_0 \]
for a.e. $\by \in \tilde{\Gamma}_1.$  Choose any $\varphi \in C^{\infty} ([0,T] \times \Gamma)$ with $\varphi \geq 0,$ $\textrm{supp}\, \varphi \subset [0,T] \times \tilde{\Gamma}_1,$ and $\int_0^T \int_{\tilde{\Gamma}_1} \varphi >0.$  The function $\varphi: [0,T] \times \Gamma \rightarrow \mathbf R$ can be lifted to an element $w \in L^2(0,T;H^1(\Omega))$ whose trace on $[0,T] \times \Gamma$  is $\varphi$. Take $\bw =w \mathbf m \in L^2(0,T;\bW(\Omega))$.  Then
\begin{equation}
\int_0^T \int_{\Gamma} \bw \cdot \bn\,dS\,dt = \int_0^T \int_{\tilde{\Gamma}_1} \varphi \mathbf m \cdot \bn\,dS\, dt \geq \delta_0 \int_0^T \int_{\tilde{\Gamma}_1} \varphi\,dS\,dt >0. \nonumber
\end{equation}
Hence the second item holds.  Finally, to prove item 3, we must show that for any $\bw \in L^2(0,T;\bW(\Omega)),$ there exist $ \bw_1 \in \bcalw_1$ and $ \bw_2 \in \bcalw_2$ such that $\bw = \bw_1 + \bw_2.$  Note that $ \bcalw_1$ is equivalent to the set of vectors in $L^2(0,T;\bW(\Omega))$ with constant divergence.  Also, $\bcalw_2$ is equivalent to the set of vectors in $L^2(0,T;\bW(\Omega))$ with normal component on $\Gamma$ equal to $0.$  Since $L^2(0,T;H_1^0) \subseteq \bcalw_2,$ select $\bw_2 \in L^2(0,T;H_0^1(\Omega))$ satisfying $\nabla \cdot \bw_2 = \nabla \cdot \bw - \frac{1}{T|\Omega|} \int_0^T \int_{\Omega} \nabla \cdot \bw.$  Set $\bw_1 = \bw -\bw_2.$  This implies that (\ref{inf-sup1-equiv}) holds and thus completes the proof of the lemma. $\Box$

We are now ready to prove the following theorem.
\begin{theorem} Assume  that  the hypotheses of Lemma 3.3 hold.   Let $F_0=f_0 I$, $ \phi_0 =\phI\, {\det F_0}$ denote an equilibrium solution satisfying (\ref{equilibrium-general}). Suppose that $\tilde\lambda$ and $\tilde\mu$  are as in (\ref{tilde-mu})-(\ref{tilde-lambda}) and satisfy (\ref{stability}).  Assume that 
for some  $T>0$,  the prescribed boundary conditions satisfy  $\tilde\uu \in H^1(0,T;H^{\frac{1}{2}}(\Gamma_0))$ 
and  $P_0\in L^2(0,T;L^2(\Gamma))$;  let   $\bu_0 \in \bW(\Omega)$  denote the prescribed   initial displacement.  Then there exists a unique weak solution $(\bu,q,p)$ to the initial boundary value problem (\ref{nabla-p})-(\ref{boundary-stress-linear})  that satisfies
\begin{eqnarray*}
&& \bu \in L^{\infty}(0,T;H^1(\Omega)), \ \bu_t \in L^2(0,T;H^1(\Omega)) \\
&& q \in L^{\infty}(0,T;L^2(\Omega)), \ q_t \in L^2(0,T;L^2(\Omega)) \\
&& p \in L^2(0,T;H^1(\Omega)).
\end{eqnarray*}
\end{theorem}
\textit{Proof}:  First of all, we note that the right hand side terms   $\mathbf f_i$, $H$ and $\mathbf {\mathcal H}$ of the governing equations are given by   (\ref{h-kappa})-(\ref{Hh}) with $P=0$ and   $Q=\nabla\cdot\uu$.
 We apply the Faedo-Galerkin method together with the discrete version of the inf-sup condition of Lemma 4.2.  For this, we  decompose $\bW(\Omega)=\bW_1(\Omega) \oplus \bW_2(\Omega)$, where $\bW_1$ is the set of all divergence-free vectors in $\bW(\Omega)$ and $\bW_2$ denotes its orthogonal complement under the inner product $\int_{\Om} \bD(\bw)_{ij}\bD(\bv_1)_{ij}$, for $\bw, \bv_1 \in \bW(\Omega)$.  $\bW_1$ and $\bW_2$ are both separable Hilbert spaces, so there exist  sequences of linearly independent smooth functions $\{ \bw^{(1),k} \}_{k=1}^{\infty}$  and $\{ \bw^{(2),k} \}_{k=1}^{\infty}$  which are dense in $\bW_1$ and $\bW_2$, respectively.  Moreover,  the sequence $\{ \bw^{(1),k}, \bw^{(2),k} \}_{k=1}^{\infty}$ forms a linearly independent dense set in $\bW(\Omega)$.  Define $\phi^k = \dive \bw^{(2),k}$ for $k=1, \ldots, \infty$.  $\{\phi^k\}_{k=1}^{\infty}$ forms a linearly independent dense set in $\Phi$.  Since $\Psi \subseteq \Phi,$ $\{\phi^k\}_{k=1}^{\infty}$ forms a linearly independent dense set in $\Psi$ as well.  For any integer $N \geq 1$, define the finite dimensional Galerkin spaces
\begin{equation}
 \bW_N = \textrm{span} \{\bw^{(1),k}, \bw^{(2),k}\}_{k=1}^N, \,\, \Phi_N = \textrm{span} \{\phi^k\}_{k=1}^N, \,\, \Psi_N = \Phi_N. \nonumber
\end{equation}
We now establish the discrete inf-sup condition.  By (\ref{inf-sup1}) and according to \cite{gatica1}, it can be shown that for the same $\alpha_0$ as in (\ref{inf-sup1}) and for each $N \geq 1$,
\begin{equation}
\sup_{\bw^N \in L^2(0,T;\bW_N)} \frac{|\int_0^T \int_{\Om} \phi^N \dive \bw^N\,d\bx\,dt|}{||\bD(\bw^N)||_{L^2(0,T;L^2(\Om))}} \geq \alpha_0 ||\phi^N||_{L^2(0,T;L^2(\Om))}, \quad \forall \phi^N \in L^2(0,T;\Phi_N). \label{inf-supN}
\end{equation}
Next, we  set up the finite dimensional approximation of the problem.  We look for $(\bu^N,q^N,p^N) \in \bW_N \times \Phi_N \times \Psi_N$  satisfying the following integral relations,  for all $(\bw^N,\phi^N,\psi^N) \in \bW_N \times \Phi_N \times \Psi_N$:
\begin{eqnarray}
&& \int_{\Om} \phi^N q^N\,d\bx = \int_{\Om} \phi^N \dive \bu^N\,d\bx, \label{imperm-weak-1N} \\
&& \int_{\Om} q_t^N \psi^N \,d\bx+\int_{\Om} \kappa \nabla p^N \cdot \nabla \psi^N\,d\bx = -\int_{\Om}\mathbf{\mathcal H} \cdot \nabla \psi^N\,d\bx, \label{imperm-weak-2N} \\
&& \int_{\Om} \big(- p^N +\tilde\lambda q^N+\mu_1q_t^N)I+ \tilde \mu \mathbf D(\bu^N)+\eta_1 \bD(\bu_t^N)\big)\cdot \mathcal D(\boldsymbol\omega^N)\,d\bx
\nonumber\\
&& \,\,\,\, =\int_{\Gamma}\mathbf G\cdot\boldsymbol\omega^N\,dS+ \int_{\Omega}(\ff_1+\ff_2)\cdot\boldsymbol\omega^N\,d\bx,
 \label{imperm-weak-3N} \\
&& \int_{\Om} \bD(\bu^N(0))_{ij} \bD(\bw^N)_{ij} \,d\bx= \int_{\Om} \bD(\bu_0)_{ij} \bD(\bw^N)_{ij}\,d\bx, \label{imperm-weak-4N} \\
&& \int_{\Om} q^N(0) \phi^N\,d\bx = \int_{\Om} q_0 \phi^N\,d\bx. \label{imperm-weak-5N}
\end{eqnarray}
This leads to a system of linear ordinary differential equations in time for the coefficients of $\bu^N, q^N,p^N$ with a complete set of initial conditions.  So,  there exists a unique triple $(\bu^N,q^N,p^N) \in \bW_N \times \Phi_N \times \Psi_N$ satisfying the system for all $t \in [0,T]$. 
As in the continuous case, it can be shown that  the discrete system  has the following energy law:
\begin{small}{
\begin{eqnarray}
&&\int_{\Om}\frac{1}{2} \big({\tilde \mu}|\bD(\bu^N(T))|^2 +{\tilde \lambda}  (q^N)^2(T)\big)d\bx+\int_0^T\int_{\Om}\frac{1}{2} \big(\kappa |\nabla p^N|^2+ {\eta_1} |\bD(\bu_t^N)|^2+\mu_1 (q^N_t)^2\big)\,d\bx\,dt\nonumber\\&& =\int_0^T\{\int_{\Gamma} \mathbf G \cdot \bu_t^N\,dS+ \int_{\Omega}(\mathbf f_1+\mathbf f_2)\cdot \bu_t^N\,d\bx\}\,dt+ \int_{\Om}\frac{1}{2} \big({\tilde \mu}|\bD(\bu^N(0))|^2 +{\tilde \lambda}  (q^N)^2(0)\big)\,d\bx.
\label{energy-law-discrete} 
\end{eqnarray}}\end{small}
Using (\ref{energy-law-discrete}), Korn's inequality,  the initial conditions (\ref{imperm-weak-4N}) and (\ref{imperm-weak-5N}), and the discrete inf-sup condition (\ref{inf-supN}), we find that
 $\bu^N$,   $q^N$,  $\bu_t^N$, $q_t^N$ and  $p^N$ are uniformly bounded in 
 $L^{\infty}(0,T;H^1(\Om))$,  $L^{\infty}(0,T;L^2(\Om))$,
 $L^2(0,T;H^1(\Om))$,  $L^2(0,T;L^2(\Om))$ and  $L^2(0,T;H^1(\Om))$, respectively. 
 Since $T$ is finite, we find upon passing to subsequences, that
\begin{itemize}
\item $\exists \bu \in H^1(0,T;H^1(\Omega)) \cap L^{\infty}(0,T;H^1(\Omega))$ such that $\bu^N \rightharpoonup 
\bu$ in $H^1(0,T; H^1(\Omega))$ and $\bu^N \stackrel{\ast}{\rightharpoonup} \bu$ in $L^{\infty}(0,T;H^1(\Omega))$;
\item $\exists q \in H^1(0,T;L^2(\Omega)) \cap L^{\infty}(0,T;L^2(\Omega))$ such that $q^N \rightharpoonup q$ in $H^1(0,T;L^2(\Omega))$ and $q^N \stackrel{\ast}{\rightharpoonup} q$ in $L^{\infty}(0,T;L^2(\Omega))$;
\item $\exists p \in L^2(0,T;H^1(\Omega))$ such that $p^N \rightharpoonup p$ in $L^2(0,T;H^1(\Omega)).$
\end{itemize}
As in  (Temam \cite{temam1}), it can be shown that the triple $(\bu,q,\op)$ is a weak solution of the system. 
Finally, uniqueness of the weak solution follows from the energy law and the inf-sup condition. $\Box$

\subsection{Fully permeable boundary}
In this section, we consider the case where the boundary of the gel is fully permeable to its surrounding inviscid solvent. The governing equations consist of  (\ref{imperm-qeqn-1})-(\ref{stress-imperm})  together with  (\ref{h-kappa})-(\ref{Hh}). 
The initial and boundary conditions are given by  (\ref{u0})-(\ref{boundary-stress-linear}) and (\ref{permeable-pure}), which  for $\eta_2=0, \nu_2=0$, the latter reduces to 
\begin{equation} p=0, \quad \textrm{on} \,\, \partial\Omega. \end{equation}
The energy law has the same expression as in the impermeable case (\ref{energy-law-1}). 
\begin{definition}
A triple $(\bu, q, p) \in \mathbf W \times L^2(\Omega) \times H_0^1(\Omega)$ is called a weak solution if for all $(\bw, \phi, \psi) \in \mathbf W \times L^2(\Omega) \times H_0^1(\Omega)$ equations  (\ref{imperm-weak-1})-(\ref{imperm-weak-5}) hold. 
\end{definition}
We now state the following theorem.
\begin{theorem}
Assume  that  the hypotheses of Lemma 3.3 hold.   Let $F_0=f_0 I$, $ \phi_0 =\phI {\det F_0}$ denote an equilibrium solution satisfying (\ref{equilibrium-general}). Suppose that $\tilde\lambda$ and $\tilde\mu$  are as in (\ref{tilde-mu})-(\ref{tilde-lambda}) and satisfy (\ref{stability}).   Let $\tilde\bU$ and $P_0$ be as in Theorem 4.3, and  $\bu_0 \in \bW(\Omega)$ denote the displacement initial condition.
Then there exists a unique weak solution $(\bu,q, p)$ of problem (\ref{nabla-p})-(\ref{boundary-stress-linear})  which satisfies
\begin{eqnarray*}
&& \bu \in L^{\infty}(0,T;H^1(\Omega)), \ \bu_t \in L^2(0,T;H^1(\Omega)), \,  p \in L^2(0,T;H^1(\Omega)),\\
&& q \in L^{\infty}(0,T;L^2(\Omega)), \ q_t \in L^2(0,T;L^2(\Omega)).
\end{eqnarray*}
\end{theorem}
\textit{Proof}: We define the function spaces
\begin{eqnarray*}
&& \mathbf{\mathcal V}_1=\{ \bw \in \bW: \dive \bw =0\}, \,\,
 \mathbf{\mathcal V}_2=\{ \mathbf 0 \} \cup \{ \bw \in \bW \setminus \mathbf{\mathcal V}_1: \bw|_{\pOm} = \mathbf 0 \}, \\
&& \quad \quad \quad \mathbf{\mathcal V}_3=\{ \mathbf 0 \} \cup [\bW \setminus (\mathbf{\mathcal V}_1 \cup \mathbf{\mathcal V}_2)]. 
\end{eqnarray*}
We point out  that  $\mathbf{\mathcal V}_1$, $\mathbf{\mathcal V}_2$, and $\mathbf{\mathcal V}_3$ are separable Hilbert spaces.  Therefore, there exist sequences $\{\bw^{(1),k}\}_{k=1}^{\infty} \subset \mathbf{\mathcal V}_1$,   $\{\bw^{(2),k}\}_{k=1}^{\infty} \subset \mathbf{\mathcal V}_2$ and  $\{\bw^{(3),k}\}_{k=1}^{\infty} \subset \mathbf{\mathcal V}_3$ of linearly independent smooth functions which are dense in $\mathbf{\mathcal V}_1$, $\mathbf{\mathcal V}_2$ and $\mathbf{\mathcal V}_3$, respectively. The sequence $\{\bw^{(1),k}, \bw^{(2),k}, \bw^{(3),k}\}_{k=1}^{\infty}$ forms a linearly independent dense set in $\bW$.  Define $\phi^{(2),k}=\dive \bw^{(2),k}$ and $\phi^{(3),k}=\dive \bw^{(3),k}$ for $k=1,\ldots,\infty$.  The sequence $\{\phi^{(2),k}, \phi^{(3),k}\}_{k=1}^{\infty}$ forms a linearly independent dense set in $\Phi$.  Since $\Psi \subset \Phi$, and it consists of functions which are zero on $\pOm$, the sequence $\{\phi^{(2),k}\}_{k=1}^{\infty}$ forms a linearly independent dense set in $\Psi$.  For any integer $N \geq 1$, we  define the finite dimensional Galerkin spaces
\begin{equation}
 \bW_N=\textrm{span} \{\bw^{(1),k},\bw^{(2),k},\bw^{(3),k}\}_{k=1}^N, \,\, \Phi_N=\textrm{span} \{\phi^{(2),k},\phi^{(3),k}\}_{k=1}^N, \,\,
 \Psi_N=\textrm{span} \{\phi^{(2),k}\}_{k=1}^N. \nonumber
\end{equation}
It is easy to check that the discrete energy law  (\ref{energy-law-discrete}) holds as well. 
By the theory of  linear differential equations, for each $N$ there exists a unique $(\bu^N,q^N,\op^N) \in \bW_N \times \Phi_N \times \Psi_N$ satisfying (\ref{imperm-weak-1N})-(\ref{imperm-weak-5N}) for all $t \in [0,T]$. 
Integrating in time over $[0,T]$ and applying the initial conditions (\ref{u0}), and using the discrete energy law, we conclude that the sequences  $\bu^N$, $q^N$,  $\bu_t^N$, $q_t^N$ and  $p^N$ are uniformly bounded in 
$L^{\infty}(0,T;H^1(\Om))$,
 $L^{\infty}(0,T;L^2(\Om))$,
 $L^2(0,T;H^1(\Om))$,
 $L^2(0,T;L^2(\Om))$ and
 $L^2(0,T;H^1(\Om))$, respectively. 
Since  $T< \infty$, passing to subsequences gives 
\begin{itemize}
\item $\exists \bu \in H^1(0,T;H^1(\Omega)) \cap L^{\infty}(0,T;H^1(\Omega))$ such that $\bu^N \rightharpoonup 
\bu$ in $H^1(0,T; H^1(\Omega))$ and $\bu^N \stackrel{\ast}{\rightharpoonup} \bu$ in $L^{\infty}(0,T;H^1(\Omega))$;
\item $\exists q \in H^1(0,T;L^2(\Omega)) \cap L^{\infty}(0,T;L^2(\Omega))$ such that $q^N \rightharpoonup q$ in $H^1(0,T;L^2(\Omega))$ and $q^N \stackrel{\ast}{\rightharpoonup} q$ in $L^{\infty}(0,T;L^2(\Omega))$;
\item $\exists p \in L^2(0,T;H^1(\Omega))$ such that $p^N \rightharpoonup p$ in $L^2(0,T;H^1(\Omega)).$
\end{itemize}
The conclusion that the triple $(\bu,q, p)$ is a unique weak solution of the system follows as in the case of impermeable boundary.

\section{Viscous solvent}
In this section, we consider the problem of a gel immersed in a viscous solvent. That is, we take the viscosity coefficients $\eta_i>0$ and $\mu_i>0$, $i=1, 2,$ in the constitutive equations of the stress.   In contrast with the case of non-viscous solvent, 
  with  scalar permeability conditions, these  are now vector relations, for impermeable as well as permeable boundary. The governing equations are
 \begin{eqnarray}
&& \dive ({\phi_0} \bu_t +(1-{\phi_0}) \bv_2) = h, \label{zero-divergence00} \\
&& \dive \mcT_1(\nablax\bu,\nablax\bu_t) -\phi_0\nabla p+\beta(\bv_2-\bu_t)=\mathbf f_1 , \label{zero-momentum11} \\
&& \dive \mcT_2(\nablax\bu, \nabla\bv_2)-(1-\phi_0)\nabla p +\beta(\bu_t-\bv_2)=\mathbf f_2, \label{zero-momentum-22} \\
&&\calT_1=  2\tilde\mu \mathbf D(\bu) + \tilde\lambda \tr{\mathbf D(\bu)} I + \eta_1\mathbf D(\bv_1)+ \mu_1(\nabla\cdot\bv_1),
\label{stress T1} \\
&&\mathcal T_2=  \eta_2\mathbf D(\bv_2)+ \mu_2(\nabla\cdot\bv_2), \label{stress T2}\end{eqnarray}
with $\tilde \mu$ and $\tilde\lambda$ as in equations (\ref{tilde-lambda-mu}), and $h, \mathbf f_1$,  $\mathbf f_2$  and $\mathbf G$ (shown below)  as in (\ref{h-kappa})-(\ref{Hh}). 
As in the case of inviscid solvent, $\uu$, $ \uu_t, $ $\vv$ and $P$ are extensions of  the  boundary data  satisfied by  the original variables $\tilde\bu, $ $\tilde\bv_1$, $\tilde \bv_2$ and $\tilde p$, and subject to compatibility conditions.  Given $\Gamma_0, \Gamma \subset  \partial \Omega$,  $\Gamma_0\cap\Gamma=\emptyset$,  and $T>0$, initial and boundary conditions are:
\begin{eqnarray}
 &&\bu|_{t=0} =\bu_0, \label{initial u} \,\,\,
\bu|_{\Gamma_0} = \mathbf 0,  \bu_0|_{\Gamma_0} = \mathbf 0, \label {displacement gamma0}\\
&&  (\mcT_1+\calT_2) \bn|_{\Gamma} = \mathbf G.\label{traction gamma1}
\end{eqnarray}
As in the previous section, permeability conditions on $\Gamma$ need to be prescribed as well. The selection of $\vv$ in (\ref{h-kappa})-(\ref{Hh})  will be made according to the  boundary permeability. 
\subsection{Impermeable boundary}
We now assume that $ \partial\Omega$ is impermeable to the solvent. Accordingly, we require that  the vectorial boundary condition (\ref{impermeable-pure}) hold.
Following Ladyzhenskaya (\cite{La69}, Ch.1, Sec. 2]), we assume that the initial displacement  $\bu_0=\bu_0(\bx)$ in (\ref{initial u}) is continuously differentiable and such that\begin{equation}
\int_{\Omega}\nabla\cdot \bu_0 =0. \label{boundary compatibility u0}
\end{equation}
 For the sake of compatibility, we define  \begin{equation}\vv=\uu_t\, \,\, \textrm{in}\,\, \Omega.  \label{compatibility gamma1}\end{equation}
The latter together with (\ref{compatibility gamma1}) imply that 
\begin{equation}  h=-\nabla\cdot(\phi_0 \uu_t+ (1-\phi_0)\vv)=0 \,\, \textrm{in } \,\, \Omega. \label{h=0}\end{equation}
The governing system consists of equations (\ref{zero-divergence00})-(\ref{boundary compatibility u0}) and   (\ref{impermeable-pure}).  It satisfies the following energy relation:
\begin{eqnarray}
&& \frac{d}{dt} \int_{\Om} \frac{1}{2}[\tilde \mu |\bD(\bu)|^2 +\tilde\lambda (\dive \bu)^2]\,d\bx +\int_{\Om} [\eta_1 |\bD(\bu_t)|^2 +\mu_1 (\dive \bu_t)^2 +\eta_2 |\bD(\bv_2)|^2 \nonumber\\
&& +\mu_2 (\dive \bv_2)^2 +\beta |\bu_t-\bv_2|^2]\,d\bx = \int_{\Gamma} \mathbf G \cdot \bu_t\,dS -\int_{\Om}\big(\nabla\ff_1\cdot D(\bu_t)+ 
\nabla \ff_2\cdot D(\bv_2)\,d\bx.  \label{energy-law-3} 
\end{eqnarray}
The function space of the problem is 
\begin{eqnarray*}
&& \mathbf{\mathfrak W} =\{ (\bw^1,\bw^2) \in H^1(\Om) \times H^1(\Om): \bw^1, \bw^2|_{\Gamma_0} = \mathbf 0; \ \bw^1-\bw^2|_{\partial\Omega} = \mathbf 0; \\
&& \dive [\phi_0\bw^1 +(1-\phi_0) \bw^2] =0 \}.
\end{eqnarray*}
\begin{definition}
A weak solution is any $(\bu,\bv_2) \in \mathbf{\mathfrak W}$ satisfying for all $(\bw^1,\bw^2)\in \mathbf{\mathfrak W}$ the equations
\begin{eqnarray}
&& \int_{\Om}\{[\tilde\mu \bD(\bu)+\tilde\lambda \dive \bu I +\eta_1 \bD(\bu_t) +\mu_1 \dive \bu_tI]\cdot \bD(\bw^1) +[ \eta_2 \bD(\bv_2) 
 +\mu_2 \dive \bv_2 I]\cdot\bD(\bw^2)\nonumber\\ && + \beta (\bu_t-\bv_2)\cdot (\bw^1-\bw^2)\}\,d\bx = \int_{\Gamma} \mathbf G \cdot \bw^1\,dS -\int_{\Om} (\nabla\ff_1\cdot\bD(\bw^1) + \nabla\ff_2\cdot \bD(\mathbf w^2))\,d\bx,\label{visc-imperm-weak-1}   \\
&& \int_{\Om} \bD(\bu(0))\cdot\bD(\bw^1)\,d\bx = \int_{\Om} \bD(\bu_0)\cdot\bD(\bw^1)\,d\bx. \label{visc-imperm-weak-2}
\end{eqnarray}
\end{definition}
We now state  the following theorem.
\begin{theorem}
Suppose that $(\phi_0, F_0)$,  $\tilde\lambda$ and $\tilde\mu$ are as in theorem 4.3. Suppose that the viscosity coefficients satisfy $\mu_i, \eta_i>0, i=1,2$. Assume that 
for some finite $T>0$, $\uu \in H^1(0,T;H^1(\Omega))$,  $\bg \in L^2(0,T;L^2(\Gamma))$ for a.e. $t \in [0,T]$,   $\bu_0 \in \mathbf H_0^1(\Omega)$, and satisfy relations
(\ref{compatibility gamma1})-(\ref{boundary compatibility u0}).
Let $\mathbf f_i, i=1,2$, $H$ and $\mathbf {\mathcal H}$ be as in (\ref {f1})-(\ref{Hh}) with $P=0$.  Then there exists a unique weak solution $(\bu,\bv_2,p)$ to the initial boundary value problem (\ref{zero-divergence00})--(\ref{boundary compatibility u0})  and (\ref{impermeable-pure}) such that
\begin{eqnarray*}
&& \bu \in H^1(0,T;H^1(\Om)), \quad \bv_2 \in L^2(0,T;H^1(\Om)), \\
&& p \in L^2(0,T;H^1(\Omega)).
\end{eqnarray*}
\end{theorem} \textit{Proof}: $\mathbf{\mathfrak W}$ is a separable Hilbert space, so there is a sequence of linearly independent smooth functions $\{ (\bw^{1,k},\bw^{2,k}) \}_{k=1}^{\infty} \subset \mathbf{\mathfrak W}$ which is dense in $\mathbf{\mathfrak W}$.  For any integer $N \geq 1$, define the finite dimensional space
\[ \mathbf{\mathfrak W}_N = \textrm{span} \{(\bw^{1,k},\bw^{2,k})\}_{k=1}^N. \]
For any integer $N \geq 1$, we seek $(\bu^N,\bv_1^N) \in \mathbf{\mathfrak W}_N$ satisfying for all $(\bw^{1,N},\bw^{2,N}) \in \mathbf{\mathfrak W}_N$ the equation
\begin{small}{
\begin{eqnarray}
&& \int_{\Om}\big((\tilde\mu\bD(\bu^N) +\tilde \lambda \dive \bu^N I +\eta_1 \bD(\bu_t^N) +\mu_1 \dive \bu_t^N I) \cdot\bD(\bw^{1,N}) +(\eta_2 \bD(\bv_2^N)
 +\mu_2 \dive \bv_2^N I)\cdot\bD(\bw^{2,N})\big)\nonumber\\ && +\int_{\Om} \beta (\bu_t^N-\bv_2^N)\cdot (\bw^{1,N}-\bw^{2,N}) = \int_{\Gamma} \mathbf G \cdot \bw^{2,N} -\int_{\Om}\nabla\ff_1\cdot\mathbf D(\bw^{1,N})+\nabla\ff_2\cdot\mathbf D(\bw^{2,N}).\label{visc-imperm-weak-2N}
\end{eqnarray}}
\end{small}
It is easy to assert that there exists a unique  $(\bu^N,\bv_2^N) \in \mathbf{\mathfrak W}_N$ for all $t \in [0,T]$.  Take $\bw^{1,N}=\bu_t^N$ and $\bw^{2,N}=\bv_2^N$.  Integrating in time over $[0,T]$ and using (\ref{visc-imperm-weak-2N}) and standard inequalities, we obtain the energy inequality
\begin{eqnarray}
&& \int_{\Om} [\tilde\mu |\bD(\bu^N(T))|^2+\tilde\lambda (\dive \bu^N(T))^2]\,d\bx +\int_0^T \int_{\Om} [\eta_1 |\bD(\bu_t^N)|^2+\frac{1}{2} \mu_1 (\dive \bu_t^N)^2\,d\bx\,dt \label{visc-imperm-ineq} \\
&& + \eta_2 |\bD(\bv_2^N)|^2 +\mu_2 (\dive \bv_2^N)^2 + \beta |\bu_t^N-\bv_2^N|^2]\,d\bx \leq C [||\bD(\bu_0)||_{L^2(\Om)}^2 +||\bg||_{L^2(0,T;L^2(\Gamma))}^2 \nonumber \\
&& +||\uu||_{H^1(0,T;H^1(\Om))}^2 +||\vv||_{L^2(0,T;H^1(\Om))}^2]. \nonumber
\end{eqnarray}
From this inequality and the fact that $T>0$ is finite,  uniform bounds
for 
 $\bu^N$  in $H^1(0,T;H^1(\Om))$, and
$\bv_2^N$  in $L^2(0,T;H^1(\Om))$ follow.
These yield  the existence of  weak limits
 $ \bu \in H^1(0,T;H^1(\Om))$ such that $\bu^N \rightharpoonup \bu$ in $H^1(0,T;H^1(\Om))$, and 
$ \bv_2 \in L^2(0,T;H^1(\Om))$ such that $\bv_2^N \rightharpoonup \bv_2$ in $L^2(0,T;H^1(\Om))$.
As in Theorems 4.3 and 4.5, they are weak solutions of the system.   Uniqueness of weak solutions is a consequence of the linearity of the problem and the energy law (\ref{energy-law-3}).  This completes the proof of the theorem.  $\Box$
\subsection{Fully permeable boundary}
We now assume that $\partial\Omega$ is fully permeable to solvent, and require the  linearized form of the boundary permeability condition (\ref{permeable-pure}) hold: 
\begin{equation} -(1-\phi_0)(p-P_0)\bn+\eta_2\mathbf D(\bv_2)+ \mu_2(\nabla\cdot\bv_2)\bn=0, \,\, \textrm {on} \,\partial\Omega, \label{viscous full permeability} \end{equation}
where $P_0$ is the hydrostatic pressure of the solvent surrounding the gel. Assuming that $P_0\in H^{\frac{1}{2}}(\partial\Omega)$, we denote  $P\in H^1(\Omega) $ its extension to the interior of the domain. The governing system consists of equations (\ref{zero-divergence00})-(\ref{stress T2}) with forcing terms obtained as in  (\ref{h-kappa})-(\ref{Hh}) by setting $\vv=0$ and letting $P$ be as previously mentioned. 
The initial and boundary conditions are as in (\ref{initial u})-(\ref{traction gamma1}) and  (\ref{viscous full permeability}).
\smallskip

\noindent
{\bf Remark.\,}An alternate choice to taking $ \vv=0$ in in (\ref{h-kappa})-(\ref{Hh}) is letting 
$\vv=-\frac{\phi_0}{1-\phi_0}\uu_t$ . This gives $\dive(\phi_0\uu_t+(1-\phi_0)\vv)=0$, and corresponds to a class of solutions with no motion of the center of mass of the gel, with only the relative velocity present. 

The  system satisfies the energy relation (\ref{energy-law-3}).
Setting the space of test functions as
\begin{equation}
\tilde\bW =\{(\bw^1,\bw^2) \in H^1(\Om) \times H^1(\Om): \bw^1, \bw^2|_{\Gamma_0} = \mathbf 0;  \, \dive [\phi_0 \bw^1 +(1-\phi_0) \bw^2] =0 \}, \nonumber
\end{equation} 
weak solutions of the system are defined by relations (\ref{visc-imperm-weak-1}) and (\ref{visc-imperm-weak-2}). 

We now state the following theorem, which proof is analogous to that of the case of impermeable boundary.
\begin{theorem}
Suppose that $(\phi_0, F_0)$,  $\tilde\lambda$ and $\tilde\mu$ are as in theorem 4.3. Suppose that the viscosity coefficients satisfy $\mu_i, \eta_i>0, i=1,2$. Then the governing system
 (\ref{zero-divergence00})-(\ref{stress T2}), with  forcing terms (\ref{h-kappa})-(\ref{Hh})  and satisfying initial and boundary conditions  (\ref{initial u})-(\ref{traction gamma1}) and (\ref{viscous full permeability}) has a unique weak solution
$(\bu,\bv_2, p)$ satisfying
$\bu \in H^1(0,T;H^1(\Om)), \,\bv_2 \in L^2(0,T;H^1(\Om)) $ and $  p \in L^2(0,T;H^1(\Omega))$.
\end{theorem}
\section{Linearization about non-spherical equilibria} 
Let us consider the governing system linearized about equilibrium solutions $(\phi_0, F_0)$ that do not necessarily correspond to dilation or compression states.  In addition, such states may not be stress free.  The next proposition establishes an energy law for such systems. 
 \begin{proposition} Suppose that the assumptions of  Proposition 3.5 hold. Let $\mathfrak C$ and $\mathfrak D$ be as in (\ref{C-general}) and  (\ref{D}), respectively. Then   smooth solutions of the system (\ref{linearized-divergence})-(\ref{linearized-stress-2}) and (\ref{chainrule-linear1}),  and  
 (\ref{hatsigma-lin})-(\ref{linear-stress-general}) satisfy the energy relation
 \begin{eqnarray}
&& \frac{d}{dt} \int_{\Om} \frac{\phI}{2}[\bigg( \mathfrak C(\nablax\bu^T F_0)\cdot(\nablax\bu^T F_0)^T + \mathfrak D(\nablax\bu;F_0)\bigg)+\pi_{1,2}^0\,\tr^2(F^{-1}_0\nablax\bu)]\,d\bx \\ &&+\int_{\Om} [\eta_1 |\bD(\bu_t)|^2 +\mu_1 (\dive \bu_t)^2\nonumber +\eta_2 |\bD(\bv_2)|^2   +\mu_2 (\dive \bv_2)^2 +\beta |\bu_t-\bv_2|^2]\,d\bx\\
&& = \int_{\Gamma} \mathbf G \cdot \bu_t\,dS -\int_{\Om}\big(\nabla\ff_1\cdot D(\bu_t)+ 
\nabla \ff_2\cdot D(\bv_2)\big)\,d\bx.  \,\,\,\quad  \Box
  \end{eqnarray}
 \end{proposition}
 Integrating the previous relation  with respect to $t$, and using the coercivity properties of $\mathfrak C$ and  $\mathfrak D$ established in Proposition 3.4,  estimates for $\bu$ follow:
 \begin{eqnarray}
&&\int_{\Omega}\frac{|\alpha_0|}{2}|\nabla\bu F^{-1}_0|^2 +C_2 \tr^2(\nabla\bu^TF_0) +C_3\tr^2(F^{-1}_0\nabla\bu)\nonumber\\
&& \leq \int_{\Om} \frac{\phI}{2}[\bigg( \mathfrak C(\nablax\bu^T F_0)\cdot(\nablax\bu^T F_0)^T + \mathfrak D(\nablax\bu; F_0)\bigg)+\pi_{1,2}^0\,\tr^2(F^{-1}_0\nablax\bu)]\,d\bx \nonumber \\
&&+\int_0^T\int_{\Om} [\eta_1 |\bD(\bu_t)|^2 +\mu_1 (\dive \bu_t)^2 +\eta_2 |\bD(\bv_2)|^2  +\mu_2 (\dive \bv_2)^2+\beta |\bu_t-\bv_2|^2]\,d\bx\,dt\nonumber \\ &&=\int_0^T\int_{\Gamma}\mathbf G \cdot \bu_t\,dS\,dt -\int_0^T\int_{\Om}\big(\nabla\ff_1\cdot D(\bu_t)+ 
\nabla \ff_2\cdot D(\bv_2))\,d\bx\,dt.
 \end{eqnarray}
 With this estimate, the well-posedness of the weak linear systems, in the cases of non-viscous as well as viscous solvent, and for all types of boundary permeability conditions follow. This allows us to extend theorems 4.3 through 5.3  to the more general case of non-spherical equilibria with possible residual stress. 
 
\section{Numerical Simulations}
We present two-dimensional numerical simulations  of the gel models previously analyzed, in the case of a viscous gel immersed in an inviscid solvent, and for, both, impermeable and fully permeable boundary. The goal is to investigate concentrations of stress that may lead to failure of the device, if critical thresholds are attained. 
We developed a fully discrete numerical method based on finite elements. 
All simulations have been performed using the DOLFIN library of the 
FEniCS project \cite{dolfin1, fenics1}.  
The equations are linearized about a  stress-free swollen or  contracted equilibrium state,  which is  consistent with  the gel having residual stress. 
We assume that this state corresponds to  that of the device previous to implantation.

The domain of the gel is the  unit square 
$\Om =[0,1] \times [0,1]$. We construct a uniform mesh of $2048$ 
triangles, each with height $h=2^{-5}$.  We take a uniform partition of the time interval  and use the 
backward Euler method to discretize the PDE system in time. 

We carry out the  non-dimensionalization  of the equations according to the following choices of scales:
\begin{itemize}
\item
 Stresses 
are normalized by the pressure scale  $\mu_E$,  the elastic modulus of the polymer (\ref{example}).
\item 
The Flory-Huggins energy density  (\ref{FH})   is  scaled by the factor $\frac{K\theta}{V_m}$ \cite{suo2}. We set $\chi=0.5$ in (\ref{FH}).
\item  We take the time scale as  $T=\frac{\eta_1}{\mu_E}$ sec, where $\eta_1$ denotes the viscosity coefficient of the polymer. We set the length scale to  $L=1 \textrm{cm}$. \end{itemize}
 We impose mixed displacement-pressure boundary conditions as explained in section 2.2.  We assume the part of the boundary $\Gamma=\{y=0\}\cup\{y=1\} $ is subject to 
 a pressure, $P_0$, that we take   to be consistent with the arterial pressure: $ P_0= 10^4$\,Pa.   Zero boundary 
displacement is imposed throughout $\Gamma_0=\{x=0\}\cup\{x=1\}$. 
  A normalized initial displacement $\bu_0=(\frac{1}{2\pi} \sin (2\pi x) 
f_0(1-f_0^3),y(1-\cos (2\pi x))f_0(1-f_0^3))$ is imposed  in $\Om$, where $f_0$ denotes an equilibrium expansion or compression.   
We compute stress components, labeling normal stresses as $\sigma_{xx}$ and $\sigma_{yy}$, and letting $\sigma_{xy}$ denote the shear stress.  
 The simulations address the following issues:
\begin{enumerate}
\item  The ratio of energy scales, $\frac{\mu_E}{K\theta/V_m}$.  We show simulations for the elastic modulus $\mu_E=10^9Pa$ which reflects values used in polymer made
devices.  The  scale of the Flory-Huggins energy is taken between 1 and $10^{-2}$.
We set  $\eta_1=10^8$ Pa$\times$sec, which results in a time scale of $0.1$ sec.
\item The degree of stiffness, expansion and compressibility of the polymer as represented by the energy exponents $s, q $ and $r$, respectively. 
\item  Type of permeability  of the boundary. 
\end{enumerate}




%
%
  \begin{figure*}
  \centerline{\includegraphics[width=2.4in]{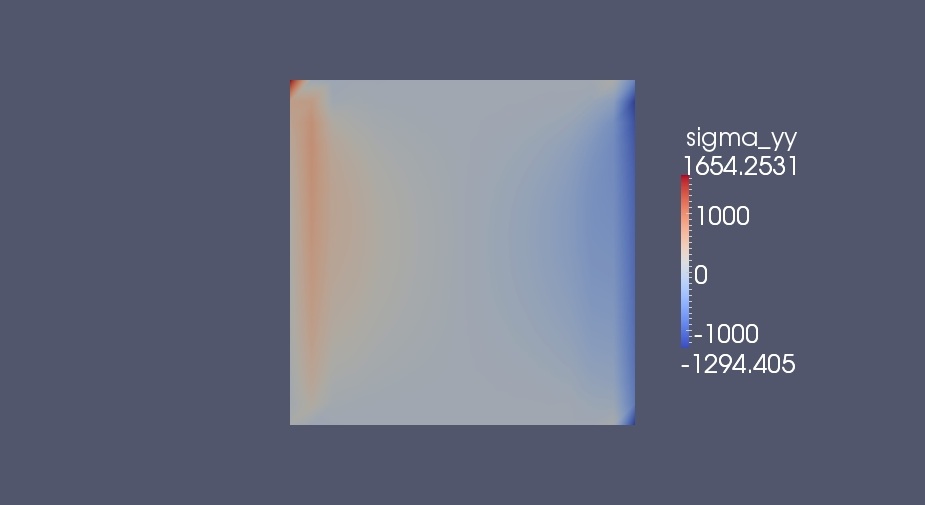}\quad
 \includegraphics[width=2.4in]{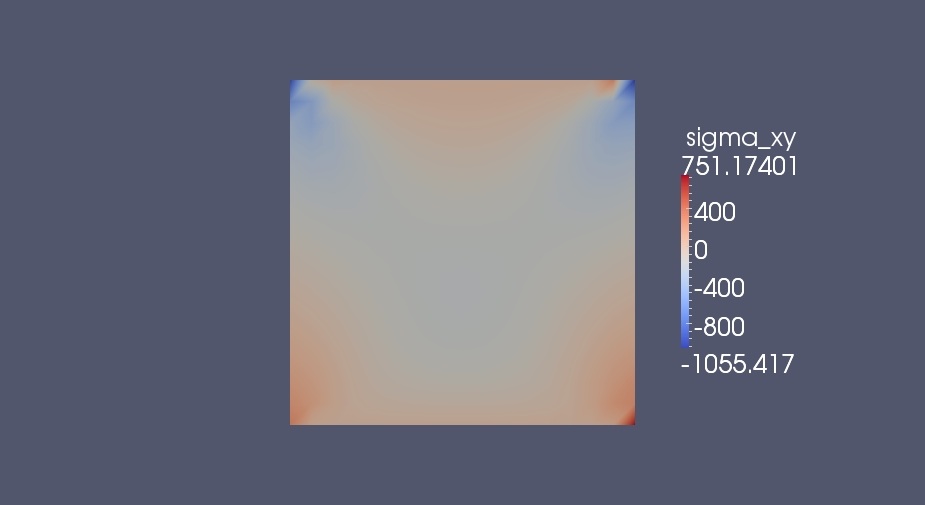}}
 \end{figure*}
 \begin{figure*}
 \centerline{\includegraphics[width=2.4in]{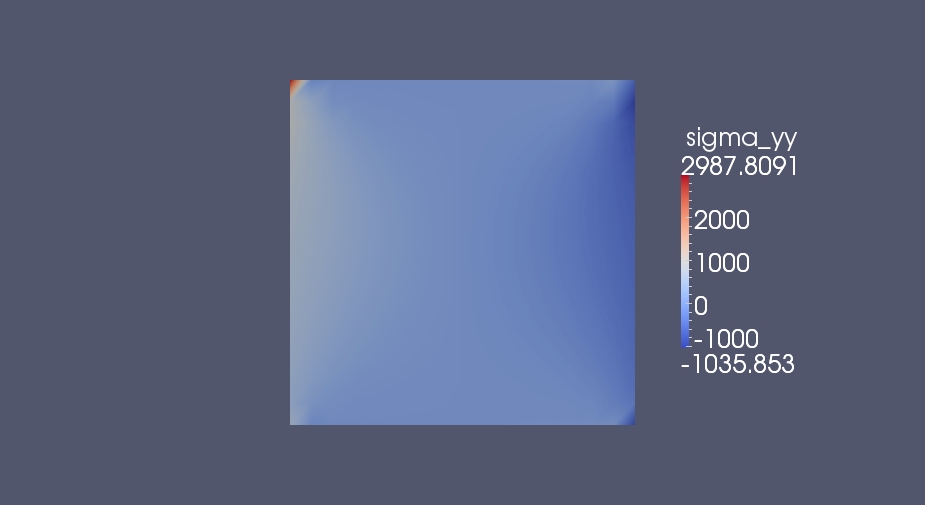}\quad
 \includegraphics[width=2.4in]{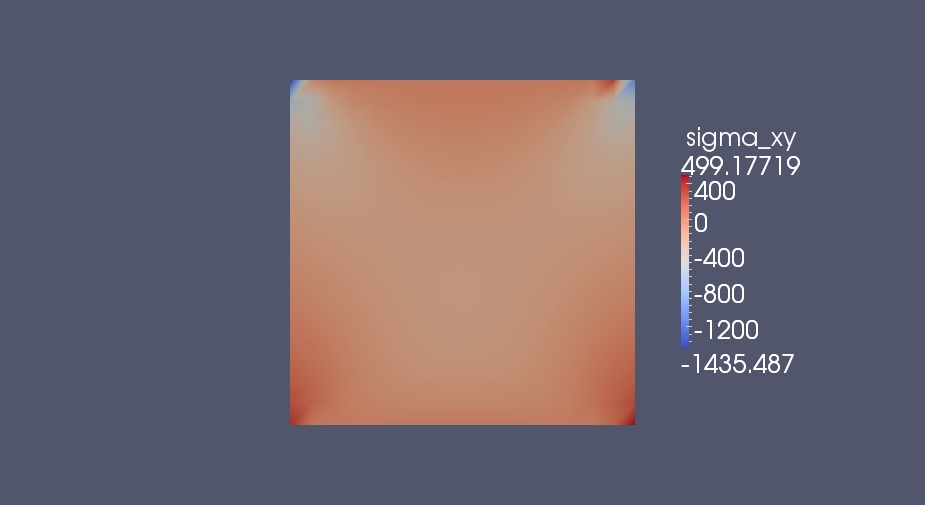}}
 \caption{{Stress components $\sigma_{yy}$ and $\sigma_{xy}$ for $\mu_E=1$GPa, parameters $s=3,
q=1.5,
r=4$ and Flory-Huggins scaling parameter $10^5$Pa. Top row corresponds to fully permeable boundary and bottom row to impermeable.}} 
 \end{figure*}

  \begin{figure*}
  \centerline{\includegraphics[width=2.3in]{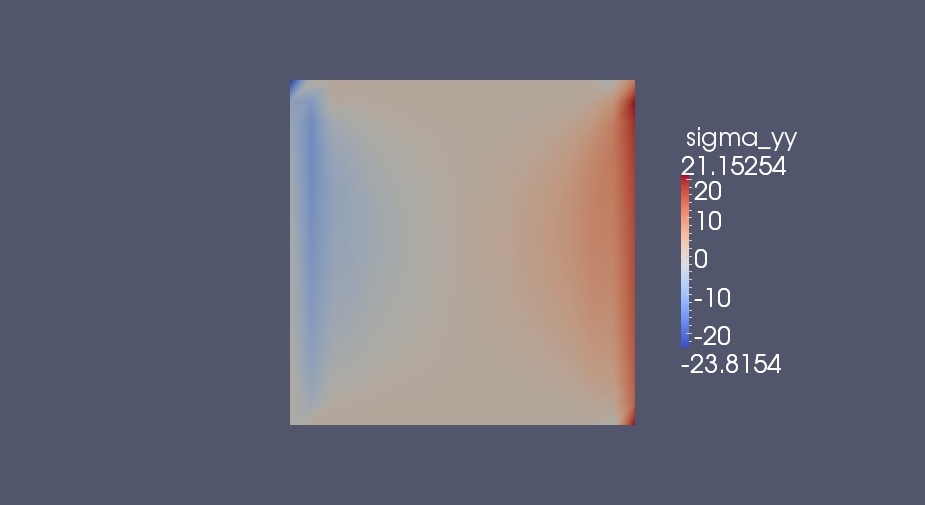}\quad
 \includegraphics[width=2.3in]{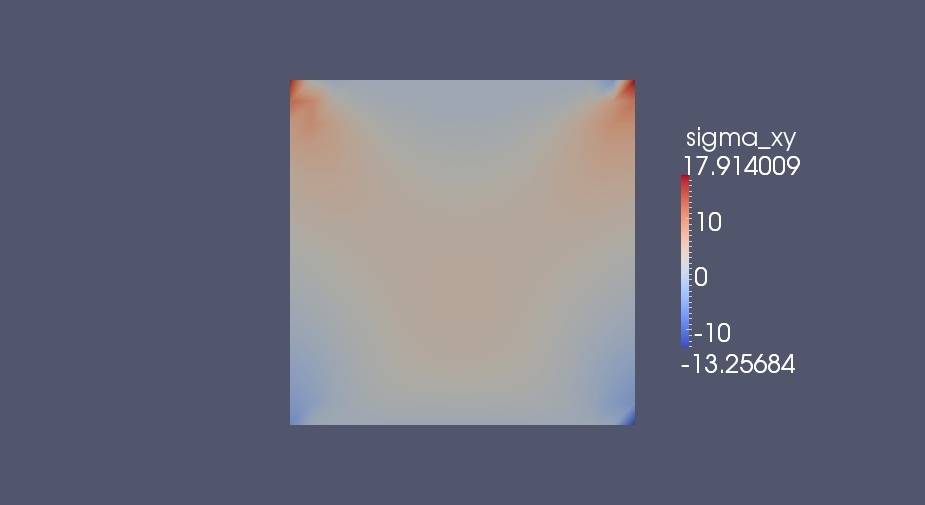}}
 \caption{{Stress components $\sigma_{yy}$ and $\sigma_{xy}$ for $\mu_E=10^9$Pa, parameters $s=1,
q=1.5,
r=1.1$ and Flory-Huggins scaling parameter $10^7$Pa;  fully permeable boundary.}} 
 \end{figure*}
We summarize the findings of our numerical simulations as follows.
\begin{enumerate}
\item The $\sigma_{yy}$ component presents  stress concentration on the two fixed displacement edges, $\Gamma_0$, for the whole range of parameters that we tested. 
 The components $\sigma_{xy}$ and $\sigma_{xx}$ (not included here) show corner concentration. In the case of impermeable boundary, the stress distributes almost uniformly
 across the domain, showing higher values than in the permeable case.  Gels with permeable boundary show a  low stress profile in the interior of the domain, with stresses concentrating on  $\Gamma_0$. 
\item  The boundary stress concentrations of the $\sigma_{yy}$ component may trigger  {\it debonding}  upon reaching a experimentally determined threshold value 
\cite{DoiYamaguchi2006}.
\item  For  a given set of parameters,   stresses in the case of impermeable boundary are higher than their permeable counterparts. This reflects the fact that, 
in a gel with fully permeable boundary, exchange of solvent  takes place across the interface causing some stress relaxation. Moreover,  in the ideal case of pure 
permeability, the fluid exchange takes place without loss of energy.
\item We have performed simulations with values of $s$ ranging from $s=1$ (Neo-Hookean material) to $s=3$ ({\it hard rubber}), and for values of $r$ 
ranging from $r=1.5 $ (high compressibility) to $r=4$. We found that 
raising either of these exponents by 1, it may increase the stresses by at least by one order of magnitude. 
\item  The stress values, as represented by their maximum and minimum absolute values,  show a decreasing pattern with the increase of the Flory-Huggins  energy
scaling 
with respect to the elastic one, reflecting softenning of the material. 
\item    The  stresss distribution shown in the figures correspond to time equal to one hour.  Calculations done for the same data after one day, show stress values in
the same order of magnitude as the ones presented here. 
\item Whereas the values of stresses shown in Figure 2 may be near the debonding pressure threshold, those in Figure 1 may have already  crossed it.  The experimental literature reports 
on values of the debonding pressure for different materials  and loading conditions ranging from 0.5 to 10 times the elastic modulus $\mu_E$  \cite{shull-creton2004}. 

\end{enumerate}


%

\section{Conclusions}

We analyzed a model of the dynamics of gels that addresses inviscid and viscous solvent and polymer,  permeability and traction-displacement boundary conditions, elasticity and diffusion.  In particular, 
we focused on the linearized system about relevant equilibrium solutions and  derived  conditions for the solvability of the time dependent problems. 
These are also conditions that guarantee local stability of the equilibrium solutions of stress-free dilation and compression states as well as general equilibria, that is, solutions of traction-displacement boundary value problems of nonlinear elasticity.  Although we  proved well-posedness of the solutions of the time dependent
 equations linearized about dilation and compression states, the energy laws that we derived  would allow us to extend the results to the more general linearized equations in a straight forward manner. In particular, the latter includes reference states with residual  stress. 


The  analysis developed in this article answers  specific questions arising in applications. Indeed, the assumptions ensuring stability of equilibrium solutions,
 and the subsequent well-posedness of the time dependent problem are formulated in terms of the parameters of the elastic and Flory-Huggins energies  and  their relative scale. 
Although this is far from sufficient to  identify a material for a specific application, it does provide a criteria to eliminate materials for which instability would occur. 
 This would have an immediate effect on reducing the number of costly and time consuming experiments to test a certain material for application by as much as 50 percent \cite{lyu1}. In addition to the stability characterization of the parameters, the numerical simulations provide data that indicate whether gel pressure has reached the debonding threshold. 
Furthermore, the choice of viscosity and drag coefficients determine the time of relaxation of a disturbance.  

The simulations presented in the paper accurately address boundary conditions encountered in device applications, as well as values of elastic modulus $\mu_E=1GPa$ and Flory-Huggins parameters of realistic device materials.  However, the domains that we use  are two-dimensional and so, cannot represent realistic shapes of devices. 
Another important feature not addressed in the current research is the stress concentration phenomenon at the interface between two different materials of the device.
 Development of numerical tools based on Discontinuous Galerkin methods is currently underway to  simulate actual devices more accurately. 
 
 The  stress corner concentrations that we found have also been observed in gel membrane experiments on drug delivery devices. In this case, though, the presence of ions significantly magnifies the effect {\cite{Siegel-personal}}.

From a different point of view, a better understanding of the debonding phenomenon is needed, perhaps appealing to the problem of cavitation and cavity propagation.  Experimental work on debonding also brings out the viscoelastic aspects of the phenomenon,  so its treatment may require the adoption of  viscoelastic stress strain laws \cite{DoiYamaguchi2006}.

The analysis presented here can be extended to treating triphasic models developed in the study of drug-delivery devices \cite{cushman2008I}, \cite{cushman2008II}. However, 
this extension is not straightforward since the  laws of balance of mass in the latter case  are significantly more challenging. 

From the point of view of analysis, one goal of the forthcoming work is to study the nonlinear problem  within the context  of the Oldroyd-B models of nonlinear elasticity.
We point out that Sections 3.1.1 and 6 deal with the linearization of the system about arbitrary equilibrium solutions. This provides  a necessary  ingredient in the time discretization of a nonlinear model.

\section{Acknowledgements} This work was partially supported by  the National Science Foundation, grant number DMS 0909165. The authors also wish to extend their appreciation
 to Medtronic, Inc., Twin Cities,  for the  financial  support and  technical advice, especially by Dr. Suping Lyu,    throughout the development of the  project. The authors also wish to thank Professors Hans Weinberger, Francisco Javier Sayas, Bernardo Cockburn and Satish Kumar for the many useful discussions.

\bibliography{gel}

\begin{thebibliography}{10}

\bibitem{dolfin1}
DOLFIN project http//www.fenics.org/dolfin.

\bibitem{Ball77}
J.~M. Ball.
\newblock Convexity conditions and existence theorems in nonlinear elasticity.
\newblock {\em Arch. Ration. Mech. Anal.}, 63:337--403, 1977.

\bibitem{cushman96a}
L.~S. Bennethum and J.~H. Cushman.
\newblock Multiscale, hybrid mixture theory for swelling systems--i: Balance
  laws.
\newblock {\em Int. J. Eng. Sci}, 34:125--145, 1996.

\bibitem{cushman96b}
L.~S. Bennethum and J.~H. Cushman.
\newblock Multiscale, hybrid mixture theory for swelling systems--ii:
  Constitutive theory.
\newblock {\em Int. J. Eng. Sci}, 34:147--169, 1996.

\bibitem{gel-bamboo-pattern03}
A.~Boudaoud and S.~Chaieb.
\newblock Mechanical phase diagram of shrinking cylindrical gels.
\newblock {\em Phys. Rev. E}, 68:021801, 2003.

\bibitem{sultan-baodoui-gelbuckling08}
A.~Boudaoud and E.~Sultan.
\newblock The buckling of a swollen thin gel layer bound to a compliant
  substrate.
\newblock {\em J.~Appl. Mech.}, 75:051002, 2008.

\bibitem{Br74}
F.~Brezzi.
\newblock On the existence, uniqueness and approximation of saddle point
  problems arising from {Lagrange} multipliers.
\newblock {\em RAIRO Numerical Analysis}, 8:129--151, 1974.

\bibitem{calderer-chabaud-zhang10}
M.~C. Calderer, B.~Chabaud, S.~Lyu, and H.~Zhang.
\newblock Modeling approaches to the dynamics of hydrogel swelling.
\newblock {\em Journal of Computational and Theoretical Nanoscience}, -7(4),
  2010.

\bibitem{CaZha07}
M.~C. Calderer and H.~Zhang.
\newblock Incipient dynamics of swelling of gels.
\newblock {\em SIAM J.~Appl.~Math.}, 68:1641--1664, 2008.

\bibitem{chabaud09}
B.~Chabaud.
\newblock Models, analysis and numerics of gels.
\newblock {\em Univeristy of Minnesota,Ph.D thesis, 2009}, 2009.

\bibitem{CI87}
P.G. Ciarlet.
\newblock {\em Mathematical Elasticity, Vol 1}.
\newblock North-Holland, 1987.

\bibitem{Dafermos95}
C.~M. Dafermos.
\newblock A system of hyperbolic conservation laws with frictional damping.
\newblock {\em Z.~Angew Math. Phys.}, 46:S294--S307, 1995.

\bibitem{Da03}
C.~M. Dafermos.
\newblock {\em Hyperbolic Conservation Laws in Continuum Physics}.
\newblock Springer, 2005.

\bibitem{DOI04}
M.~Doi and A.~Onuki.
\newblock Dynamic coupling between stress and composition in polymer solutions
  and blends.
\newblock {\em J. Phys. II France}, 2:1631--1656, 1992.

\bibitem{fenghe}
X.~Feng and Y.~He.
\newblock Analysis of fully discrete finite element methods for a system of
  differential equations modeling swelling dynamics of polymer gels.
\newblock Submitted, 2009.

\bibitem{flory}
P.J. Flory.
\newblock {\em Principles of Polymer Chemistry}.
\newblock Cornell U. Press, 1953.

\bibitem{GAS}
D.~R. Gaskell.
\newblock {\em Introduction to the Thermodynamics of Materials}.
\newblock Taylor \& Francis, 1995.

\bibitem{gatica1}
G.~Gatica and F.J. Sayas.
\newblock Characterizing the inf-sup condition on product spaces.
\newblock {\em Numer. Math.}, 109:209--231, 2008.

\bibitem{GiRa79}
V.~Girault and P.~A. Raviart.
\newblock {\em Finite Element Approximation of the Navier Stokes Equations}.
\newblock Number 749 in Lecture Notes in Mathematics. Springer Verlag, Berlin,
  Heidelbert, New York, 1979.

\bibitem{chen-mori-micek-calderer2012}
K.Micek H.Chen, Y.Mori and M.C.Calderer.
\newblock A dynamic model of polyelectrolyte gels.
\newblock {\em SIAM J.Appl. Math}, in press, 2012.

\bibitem{fenics1}
FEniCS~Project. http://www.fenics.org.
\newblock University of Chicago, Chalmers University and University of Oslo.

\bibitem{KAISER03}
D.~Kaiser.
\newblock Coupling cell movement to multicellular development in myxobacteria.
\newblock {\em Nature Reviews Microbiology}, 1:45--54, 2003.

\bibitem{shull-creton2004}
K.R.Shull and C.Creton.
\newblock Deformation behavior of thin, compliant layers under tensile loading
  conditions.
\newblock {\em J. Polymer Sci. Prg B: Polym, Phys.}, 42:4023--4073, 2004.

\bibitem{La69}
O.~A. Ladyzhenskaya.
\newblock {\em The Mathematical Theory Of Viscous Incompressible Fluid}.
\newblock Gordon and Breach, 1969.

\bibitem{lyu1}
S.~Lyu.
\newblock Personal communication, 2010.

\bibitem{rognes-calderer-micek09}
M.C.~Calderer M.~Rognes and C.~Micek.
\newblock Mixed finite element methods for gels with biomedical applications.
\newblock {\em SIAM J.~Appl.Math}, 70:1305--1329, 2009.

\bibitem{Gurtin-Fried10}
E.~Fried M.E.~Gurtin and L.~Anand.
\newblock {\em Continuum Mechanics and Thermodynamics}.
\newblock Cambridge University Press, 2009.

\bibitem{cushman00}
M.~A. Murad, L.~S. Bennethum, and J.~H. Cushman.
\newblock Macroscale thermodynamics and the chemical potential of swelling
  porous media.
\newblock {\em Transport in Porous Media}, 39:187--225, 2000.

\bibitem{Onuki89}
A.~Onuki.
\newblock Theory of pattern formation in gels: Surface folding in highly
  compressible elastic bodies.
\newblock {\em Phys Rev A}, 39:5932--5948, 1989.

\bibitem{RD81}
H.~Reichenbach and M.~Dworkin.
\newblock Introduction to the gliding bacteria.
\newblock In {\em The prokaryotes}, pages 315--327, 1981.

\bibitem{sayas1}
F.J. Sayas.
\newblock Personal communication, 2010.

\bibitem{Siegel-personal}
R.A. Siegel.
\newblock Personal communication, 2011.

\bibitem{temam1}
R.~Temam.
\newblock {\em The Navier-Stokes Equations: Theory and Numerical Analysis}.
\newblock 2nd edn. North-Holland, Amsterdam, 1977.

\bibitem{cushman2008I}
J.~H.~Cushman T.J.~Weinstein and L.~S. Bennethum.
\newblock Two-scale, three-phase theory for swelling drug delivery systems.
  part i: Mixture theory.
\newblock {\em J.Pharm Sci}, 97:1878--1903, 2008.

\bibitem{cushman2008II}
J.~H.~Cushman T.J.~Weinstein and L.~S. Bennethum.
\newblock Two-scale, three-phase theory for swelling drug delivery systems.
  part ii: Flow and transport.
\newblock {\em J.Pharm Sci}, 97:1904=1915, 2008.

\bibitem{TruesdellNoll2010}
C.~Truesdell and W.Noll.
\newblock {\em The Non-Linear Field Theories of Mechanics}.
\newblock Springer Verlag, third edition, 2010.

\bibitem{suo2}
J.Zhou W.Hong, X.~Zhao and Z.Suo.
\newblock A theory of coupled diffusion and large deformation in polymeric
  gels.
\newblock {\em J.Mech.Phys. Sol.}, 56:1779--1793, 2008.

\bibitem{DoiYamaguchi2006}
T.~Yamaguchi and M.~Doi.
\newblock Debonding dynamics of pressure-sensitive adhesives: 3d block model.
\newblock {\em Eur Phys J.E}, 21:331--339, 2006.

\bibitem{DOI02}
T.~Yamaue and M.~Doi.
\newblock Theory of one-dimensional swelling dynamics of polymer gels under
  mechanical constraint.
\newblock {\em Phys. Rev. E}, 69:041402, 2004.

\bibitem{DOI05}
T.~Yamaue, H.~Mukai, K.~Asaka, and M.~Doi.
\newblock Electrostress diffusion coupling model for polyelectrolyte gels.
\newblock {\em Macromolecules}, 38:1349--1356, 2005.

\bibitem{DOI06}
T.~Yamaue, T.~Taniguchi, and M.~Doi.
\newblock The simulation of the swelling and deswelling dynamics of gels.
\newblock {\em Molecular Physics}, 102(2):167--172, 2004.

\end{thebibliography}
\bibliographystyle{plain}

\end{document}